\documentclass[12pt,a4paper]{amsart}
\usepackage[utf8]{inputenc}
\usepackage[english]{babel}
\usepackage{lmodern}
\usepackage{array,booktabs}
\usepackage{multicol}
\usepackage{mathrsfs, amssymb, amsmath, amsfonts}
\usepackage{graphicx}
\usepackage{geometry, xcolor}
\usepackage{enumitem}
\usepackage[abs]{overpic}
\usepackage[pagewise]{lineno} 
\allowdisplaybreaks[3]

\title{On the Optimality and Decay of $p$-Hardy Weights on Graphs}

\author{Florian Fischer}
\address{Florian Fischer, Institute for Applied Mathematics, University of Bonn, Endenicher Allee 60, 53115 Bonn, Germany}
\email{fischer@iam.uni-bonn.de}

\usepackage[style=alphabetic,maxnames=5,minnames=2,abbreviate=false,date=long,url=false,isbn=false,firstinits, doi=false,backend=bibtex, backref,backrefstyle=none]{biblatex}
\usepackage{csquotes}
\AtBeginBibliography{\tiny}
\usepackage[colorlinks=true,citecolor=blue,linkcolor=blue]{hyperref}

\newtheorem{theorem}{Theorem}[section]
\newtheorem{lemma}[theorem]{Lemma}
\newtheorem{proposition}[theorem]{Proposition}
\newtheorem{corollary}[theorem]{Corollary}

\theoremstyle{definition}
\newtheorem{example}[theorem]{Example}
\newtheorem*{remark}{Remark} 
\newtheorem{definition}[theorem]{Definition}

\numberwithin{equation}{section}

\newcommand{\norm}[1]{\left\lVert #1 \right\rVert} 
\newcommand{\abs}[1]{\left\lvert #1\right\rvert} 
\newcommand{\set}[1]{\left\{ #1\right\} }
\newcommand{\ip}[2]{\left\langle #1, #2 \right\rangle}
\newcommand{\p}[1]{\left( #1 \right)^{\langle p-1 \rangle}}
\newcommand{\sse}{\subseteq}

\renewcommand{\epsilon}{\varepsilon}
\renewcommand{\phi}{\varphi}
\newcommand{\NN}{\mathbb{N}}
\newcommand{\ZZ}{\mathbb{Z}}
\newcommand{\RR}{\mathbb{R}}

\newcommand{\TT}{\mathbb{T}}
\newcommand{\HHH}{\mathbb{H}}
\newcommand{\dd}{\mathrm{d}}
\DeclareMathOperator{\sgn}{sgn}

\DeclareMathOperator{\supp}{supp}

\newcommand{\FF}{F}
\newcommand{\KK}{\tilde{K}}

\newcommand{\MM}{\mathcal{M}}
\newcommand{\hh}{\tilde{h}}
\newcommand{\ww}{\tilde{w}}

\newcommand{\Hmm}[1]{\leavevmode{\marginpar{\tiny%
			$\hbox to 0mm{\hspace*{-0.5mm}$\leftarrow$\hss}%
			\vcenter{\vrule depth 0.1mm height 0.1mm width \the\marginparwidth}%
			\hbox to 0mm{\hss$\rightarrow$\hspace*{-0.5mm}}$\\\relax\raggedright #1}}}
\begin{document}
\begin{abstract}
We construct optimal Hardy weights to subcritical energy functionals $h$ associated with quasilinear Schrödinger operators on infinite graphs. Here, optimality means that the weight $w$ is the largest possible with respect to a partial ordering, and that the corresponding shifted energy functional $h-w$ is null-critical. Moreover, we show a decay condition of Hardy weights in terms of their integrability with respect to certain integral weights. As an application of the decay condition, we show that null-criticality implies optimality near infinity. We also briefly discuss an uncertainty-type principle, a Rellich-type inequality and examples.\\
	\\[4mm]
	\noindent  2020  \! {\em Mathematics  Subject  Classification.}
	Primary  \! 39A12; Secondary  31C20; 31C45; 35R02.
	\\[4mm]
\end{abstract}

\maketitle


\section{Introduction}
In a letter to Hardy in 1921,  \cite{Landau21}, Landau gave a proof of the following inequality
\begin{align*}
\sum_{n=1}^\infty \abs{\phi(n)-\phi(n-1)}^p \geq C_p \sum_{n=1}^\infty w_p^{H}(n){\abs{\phi(n)}^p},
\end{align*}
for all $\phi\in C_c(\NN), \phi(0)=0$, and  $p\in (1,\infty)$, where
\[w_p^{H}(n):= \frac{1}{n^p},\quad n\in \NN, \qquad \text{and}\qquad C_p:=\left(1-\frac{1}{p}\right)^p.\]

Landau also proved that the constant $C_p$ is sharp. This inequality was stated before by Hardy, though hidden in a proof of Hilbert's double series theorem, see \cite{Hardy20}. Therefore, it got the name \emph{Hardy inequality} with \emph{Hardy weight} $C_p\cdot w_p^{H}$.

Recently, in \cite{FKP}, it was shown that $C_p\cdot w_p^{H}$ can be improved by
\begin{align*}
	w_p^{\mathrm{opt}}(n):=\left( 1-\left(1-\frac{1}{n}\right)^{\frac{p-1}{p}}\right)^{p-1}-\left( \left(1+\frac{1}{n}\right)^{\frac{p-1}{p}}-1\right)^{p-1}, \qquad n\in \NN, 
	\end{align*}
 and thus, the question of finding the `largest' possible Hardy weight occurs. As a special case of  our main result (Theorem~\ref{thm:mainresultc<0}), we show that the Hardy weight $w_p^{\mathrm{opt}}$ obtained in \cite{FKP} is \emph{optimal} and cannot not be improved.  We remark that the corresponding classical $p$-Hardy weight on $[0,\infty)$ is already optimal, see \cite{DP16}.
 
The original Hardy inequality was generalised to various contexts. In an abstract way, it can be stated as
\[E \geq W,\]
where $E$ is an energy functional associated with a non-negative operator, and $W$ is the canonically obtained non-negative functional from a positive weight function $w$. For a detailed analysis, we refer to the monographs \cite{BEL15, KPS17, OK90}, and references therein. The aim is now to make $w$ as large as possible with large support, i.e., to make $w$ optimal in a certain sense. This was proposed first in \cite[page 6]{Ag82}. 

In \cite{DFP}, an \emph{optimal Hardy weight} $w$ associated with a linear Schrödinger operator on domains in $\RR^d$ (or on non-compact Riemannian manifolds) was defined first. Roughly speaking, $w$ is optimal if
\begin{itemize}
\item $w$ is critical, i.e., for all $\ww \gneq w$ the Hardy inequality fails,
\item $w$ is null-critical, i.e., the corresponding ground state is not an eigenfunction, and
\item $w$ is optimal near infinity, i.e., for any $\lambda >0$, the Hardy inequality outside of any compact set fails for the weight $(1+\lambda)w$.
\end{itemize}

In \cite{DFP} also a way of obtaining optimal Hardy weights was given. Using a different approach, the main result and definition of optimality of \cite{DFP} were generalised to $p$-Laplacians, $p\in (1,\infty)$, on Riemannian manifolds in \cite{DP16}. Recently, for a large class of potentials (including non-positive potentials), optimal Hardy weights for $p$-Schrödinger operators were  constructed in \cite{Versano}.

On weighted locally finite graphs, inspired by the approaches from \cite{DFP, DP16}, a way of obtaining optimal Hardy weights for linear Schrödinger operators with arbitrary potential parts was given in \cite{KePiPo2}. 

In the present paper, we show in Theorem~\ref{thm:mainresultc<0} how to obtain optimal Hardy weights for $p$-Schrödinger operators with arbitrary potential term on weighted graphs which are essentially locally finite, following partially the approaches in \cite{DP16, KePiPo2, Versano}. To be more specific, we evolve discrete quasilinear versions of them, to prove criticality and null-criticality. The main tools in these proofs in \cite{DP16, KePiPo2, Versano} were a coarea formula and the ground state representation. Corresponding discrete quasilinear versions of both will also be of fundamental importance here.

For proving optimality near infinity, we establish instead discrete versions of results in \cite{KP20}, Theorem~\ref{thm:decay} and Theorem~\ref{thm:decayImpliesOpti}. Specifically, Theorem~\ref{thm:decay} states a necessary decay condition of Hardy weights in terms of a certain $\ell^{1}$-summability with respect to weights which are related to positive superharmonic functions. These results are also valid for arbitrary potential parts.

By a supersolution construction technique, we show how to obtain optimal Hardy weights for $p$-Schrödinger operators. This can be seen as a discrete version of results in \cite{Versano}. The results in the continuum made use of the chain rule, which does not hold on graphs in general. We use instead the mean value theorem to circumvent the problem.

The outline of the present paper is as follows: In Section~\ref{sec:setting}, the basic setting is introduced and the main results are stated. Thereafter, in Section~\ref{sec:decay}, we first prove some convexity-type results which are discrete versions of results in \cite{HPR21, PR15, PP, PT07} and state the ground state representation formula, confer \cite{F:GSR}. Then, we turn to the proofs of Theorem~\ref{thm:decay} and Theorem~\ref{thm:decayImpliesOpti}. In Section~\ref{sec:opti}, we prove a supersolution construction and a coarea formula. Thereafter, we  show a proof of  Theorem~\ref{thm:mainresultc<0}. In Section~\ref{appli}, we show two applications of the Hardy inequality, the Heiseberg-Pauli-Weyl uncertainty principle and a Rellich-type inequality. We close this paper by  discussing briefly some examples. In an appendix, we show that criticality implies the existence of increasing null sequences.

\section{Preliminaries and Main Results}\label{sec:setting}

\subsection{Graphs and Schrödinger Operators}
Let $X$ be a countable infinite set. By a \emph{(weighted) graph} $b$ on $X$ we understand  a symmetric function  $b\colon X\times X \to [0,\infty)$ with zero diagonal such that $ b $ is locally summable, i.e., the vertex degree satisfies \[ \deg(x):=\sum_{y\in X}b(x,y)<\infty, \qquad x\in X.\] The elements of $X$ are called \emph{vertices} and the function $b$ decodes weights between vertices. Two vertices $x, y$ are called \emph{connected} with respect to the graph $b$ if $b(x,y)>0$, in terms $x\sim y$. A set $V\sse X$ is called \emph{connected} with respect to $b$, if for every two vertices $x,y\in V$ there are vertices ${x_0,\ldots ,x_n \in V}$, such that $x=x_0$, $y=x_n$ and $x_{i-1}\sim x_i$ for all $i\in\set{1,\ldots, n-1}$. For $V\sse X$ let $\partial V=\set{y\in X\setminus V : y\sim z\in V}$ denote the exterior boundary of $V$.
Throughout this paper we will always assume that 
\begin{center}$X$ is connected with respect to $b$.\end{center}

A graph $b$ on $X$ is called \emph{locally finite} if for all $x\in X$ 
\[\# \set{y\in X : y\sim x  }< \infty.\]

Let $S$ be some arbitrary set, then $f\colon S\to \RR$ is called \emph{non-negative}, \emph{non-positive}, \emph{positive}, or \emph{strictly positive} on $I\sse S$, if $f\geq 0$, $f\leq 0$, $f\gneq 0$, $f>0$ on $I$, respectively. We set $f_-=(-f)\vee 0$. If for two non-negative functions $f_1, f_2\colon S\to \RR $ there exists a constant $C>0$ such that $C^{-1}f_1 \leq f_2 \leq Cf_1$ on $I\subset S$, we write \[f_1 \asymp f_2 \quad\text{ on }I,\]
and call them \emph{equivalent} on $I$. 

 The space of real valued functions on $V\subseteq X$ is denoted by $C(V)$ and the space of functions with compact support in $V$ is denoted by $ C_c(V)$. We consider $C(V)$ to be a subspace of $C(X)$ by extending the functions of $C(V)$ by zero on $X\setminus V$.  
 
We also introduce the linear difference operator $\nabla$ on $C(X)$ defined via
\[\nabla_{x,y}f:=f(x)-f(y), \qquad x,y\in X, f\in C(X).\]

Now, we turn to Schrödinger operators: Let $p\geq 1$. For $V\sse X$ let the \emph{formal space} $ \FF(V):=\FF_{b,p}(V) $ be given by
\begin{align*}
\FF(V):= \{ f\in C(X): \sum_{y\in X} b(x,y)\abs{\nabla_{x,y}f}^{p-1} < \infty  \mbox{ for all } x\in V  \}.
\end{align*}
We set $F:=F(X)$. Note that on locally finite graphs, we have $F=C(X)$.

For $1\leq p<2$ we make the convention that $\infty\cdot 0 =0$. Then, for all $p\geq 1$ and $a\in \RR$, we set
\[ \p{a}:= |a|^{p-1} \sgn (a)=|a|^{p-2} a. \]
Here, $\sgn\colon \RR\to \set{-1,0,1}$ is the sign function, that is, $\sgn(\alpha)=1$ for all $\alpha> 0$, $\sgn(\alpha)=-1$ for all $\alpha< 0$, and $\sgn(0)=0$.

Let $p\geq 1$ and $c, m\in C(X)$ and $m>0$ on $V$, then the quasilinear \emph{(formal) ($p$-)Schrödinger operator} $H:=H_{b,c,p,m}\colon \FF(V)\to C(V)$ is given by
\[Hf(x):=Lf(x)+\frac{c(x)}{m(x)}\p{f(x)}, \qquad x\in V,\]
where $L:=L_{b,p, m} \colon \FF(V)\to C(V)$  defined via
\[ Lf(x):=\frac{1}{m(x)}\, \sum_{y\in X} b(x,y)\p{\nabla_{x,y}f}, \qquad x\in V,\]
is called the \emph{(formal) ($p$-)Laplacian} and the function $c$ is usually called \emph{potential} of $H$. If $c\geq 0$, then Schrödinger operator sometimes also go under the name \emph{Laplace-type operators}. The function $m$ represents weights on the vertices.

A function $u\in \FF(V)$ is said to be \emph{($p$-)harmonic, (($p$-)superharmonic, strictly ($p$-)superharmonic, ($p$-)subharmonic}) on $V\sse X$ with respect to $H$ if \[Hu=0 \quad(Hu\ge 0,\,Hu \gneq 0, \, Hu\leq 0)\qquad\text{ on }V.\] 
If $V=X$ we only speak of super-/sub-/harmonic functions.

We also need the following definitions: Let $V\sse X$ be connected and $K\sse V$ be finite. By $\MM(V\setminus K)\sse \FF(V\setminus K)$, we denote the set of strictly positive functions $u$ which are $p$-harmonic on $V\setminus K$, and which have the following minimal growth property: for any  finite and connected subset $\KK \sse V$ with $K\sse \KK$, and any strictly positive function $v\in \FF(V\setminus \KK)$ which is $p$-superharmonic in $V\setminus \KK$, we have
\[u\leq v \text{ on } \KK \quad\text{ implies }\quad  u\leq v \text{ in } V\setminus \KK.\]
	If $u\in \MM(V)$, then $u$ is called a \emph{global minimal positive harmonic function} in $X$, or the \emph{Agmon ground state}, see \cite{F:AAP} for details on these two notions.
	
	If for some fixed vertex $o\in V$ a function $g_o\in \MM(V\setminus \set{o})\cap \FF(V)$ satisfies $Hg_o=1_o$ on $V$, then $g_o$ is called \emph{(minimal positive) Green's function} with reference point $o\in V$. 

\subsection{Energy Functionals Associated with Graphs}
Let $p\geq 1$ and $c\in C(X)$ arbitrary. The \emph{(p-)energy functional} $h=h_{b,c,p}\colon C_c(X)\to \RR$ is defined via 
\[ h(\phi):=\frac{1}{2}\sum_{x,y\in X} b(x,y)\abs{\nabla_{x,y}\phi}^p+ \sum_{x\in X}c(x)\abs{\phi(x)}^p.\]

The double sum part in the definition of $h$ is sometimes called kinetic energy functional and the part of the second sum,  potential energy functional. If $p=2$ such a functional is a quadratic form, and then sometimes called Schrödinger form. 

The connection between $H$ and $h$ on $C_c(X)$ can be described via a Green's formula which is stated next, for a proof see \cite{F:GSR}.

Let $V\sse X$. To shorten notation, we define a weighted bracket $\ip{\cdot}{\cdot}_{V}$ on $C(X)\times C_c(X)$ via
\[\ip{f}{\phi}_{V}:=\sum_{x\in V}f(x)\phi(x)m(x),\qquad f\in C(X), \phi\in C_c(X).\]
 
\begin{lemma}[Green's formula, Lemma~2.3 in \cite{F:GSR}]\label{lem:GreensFormula}
	Let $p\geq 1$ and $V\sse X$. Let $f\in \FF(V)$ and $\phi\in C_c(X)$. Then all of the following sums converge absolutely and
	\begin{align*}
		\ip{Hf}{\phi}_{V}
		&=\frac{1}{2}\sum_{x,y\in V} b(x,y)\p{\nabla_{x,y}f}( \nabla_{x,y}\phi)+ \sum_{x\in V}c(x)\p{f(x)}\phi(x)\\
		&\qquad +\sum_{x\in V, y\in \partial V}b(x,y)\p{\nabla_{x,y}f}\phi(x). 
	\end{align*}
	In particular, the formula can be applied to $f\in C_c(X)$, and
	\[h(\phi)=\ip{H\phi}{\phi}_{V}, \qquad \phi\in C_c(V).\]
\end{lemma}
Moreover, another nice connection between $h$ and $H$ is that $p\cdot H$ is the G\^{a}teaux derivative of $h$ on $C_c(X)$, see \cite{F:AAP} for details.

If the functional $h$ is non-negative on $C_c(V)$, $V\sse X$, then $h$ is called \emph{($p$-)subcritical} in $V$ if the \emph{Hardy inequality} holds true, that is, there exists a positive function $w$ such that \[ h \geq \norm{\cdot}_{p,w}^{p} \qquad \text{on } C_c(V).\]
Here, \[\norm{\phi}_{p,w}:=\biggl(\sum_{x\in X}\abs{\phi(x)}^pw(x)\biggr)^{1/p}, \qquad \phi\in C_c(X).\]

If such a $w$ does not exist, then $h$ is called \emph{($p$-)critical} in $V$. Moreover, $h$ is called \emph{($p$-)supercritical} in $V$ if $h$ is not non-negative on $V$. 
	
\begin{remark}
This classification of energy functionals in terms of sub-/super-/critical goes back to \cite{Si0}, see also \cite{Mur84, Mur86, P88}, and is motivated from the analysis of the energy functional $h$.

In the special case of non-negative potentials and $p=2$, a graph is usually called \emph{transient} if $h$ is subcritical in $X$, and \emph{recurrent} otherwise, see e.g. \cite{KLW21}. This notation has its origin in probability theory, see e.g. \cite{WoessMarkov}, and goes back at least to \cite{Polya}. It is also common when studying Dirichlet forms, see e.g. \cite{Fuk, Stu94}, whereas subcritical and critical is used for Schrödinger forms, see \cite{Tak14, Tak16, Tak23, TU21, M23, S22}. In a hand waving way, recurrence means that the associated random walker returns almost surely, and transience means that the random walker escapes any finite set with positive probability.

Moreover, if $c=0$, then a graph is sometimes called \emph{$p$-hyperbolic} if $h$ is $p$-subcritical in $X$, and \emph{$p$-parabolic} otherwise. This notation has its origin in the geometry of surfaces, see here e.g. \cite{Gri99, Y84, SY93Class, SY93Para, PRS05, PS14, T00, T99, Y77}.
\end{remark}	
	
It is shown in \cite{F:AAP}, that $h$ is critical if and only if there exists a unique positive superharmonic function (up to multiplies). This function is harmonic and the Agmon ground state of $h$.

Let $h$ be a critical energy functional. We call $h$ \emph{null-critical}  with respect to the non-negative function $w$ in $V$ if the Agmon ground state is not in $\ell^p(V,w)$, and otherwise we call it \emph{positive-critical} with respect to $w$ in $V$. 

Here, for all $1\leq p < \infty$, $V\sse X$ and non-negative functions $w$ on $V$, we define 
\begin{align*}
\ell^p(V,w)&:=\bigl\{f\in C(X): \sum_{x\in V}\abs{f(x)}^pw(x)<\infty\bigr\},\quad \text{and} \\
\ell^{\infty}(V)&:=\bigl\{f\in C(X): \sup_{x\in V}\abs{f(x)}<\infty\bigr\}.
\end{align*}
Note that $(\ell^p(X,w),\norm{\cdot}_{p,w})$ is a reflexive Banach space for $p\in (1,\infty)$, and a Banach space for $p=1$. 

The focus of this paper is on estimates of functionals. It is therefore comfortable to use the following notation: Any function $w\in C(X)$ gives rise to a canonical $p$-functional $w_p$ on $C_c(X)$ via
\[w_p(\phi):=\norm{\phi}^p_{p,w}= \sum_{x\in X}\abs{\phi(x)}^pw(x), \qquad \phi\in C_c(X).\]
The Hardy inequality then reads as $h-w_p\geq 0$ on $C_c(X)$ for some $w\gneq 0$ on $X$. If it is clear that we mean the functional $w_p$ and not the function $w$, we sometimes simply write $w$ instead of $w_p$.

If not stated otherwise, we will always assume that 
\[p\in (1,\infty).\]

Moreover, we always explicitly mention the assumption that the underlying graph is locally finite.

\subsection{Main Results}
We have three main results and in the first two a positive superharmonic function has to satisfy two properties, which are defined next: A function $f\in C(X)$ that is positive on $V\sse X $  is called \emph{almost proper} on $V$ if $f^{-1}(I)\cap V$ is a finite set for any compact set $I\sse (0,\infty)$. Note that this definition differs slightly from standard definitions of a proper function as we allow $u$ to vanish infinitely often on $V$. Moreover, a strictly positive and almost proper function on $V$ is called \emph{proper} on $V$. If $f$ is a proper function on $V$, then $ 0$ and $ \infty$ are the only possible accumulation points of $f|_V$, and if $V$ is also infinite at least one of them is always an accumulation point.

Moreover, $f\in C(X)$ is called \emph{of bounded oscillation} on $V$ if \[\sup_{x,y\in V, x\sim y} \abs{f(x)/f(y)}<\infty.\] In particular, such a function cannot vanish on $V$.

\begin{remark}[Locally finiteness]
	If there exists an almost proper function of bounded oscillation on $V\cup \partial V\sse X$ (where $\partial V$ might be empty), then the graph is locally finite on $V$. 
	
	This can be seen as follows: Assume that $f$ is such a function. First note that $f>0$ on $V\cup \partial V$ since $f$ is of bounded oscillation. Then, being proper on $V\cup \partial V$ implies that $0$ or $\infty$ are the only accumulation points of $f|_{V\cup \partial V}$. Being of bounded oscillation implies that neighbouring vertices never reach an accumulation point. Thus, there exists a compact set in $(0,\infty)$ containing all the images of $f|_{V\cup \partial V}$ of neighbours of a vertex and by the properness this implies that any vertex can only have finitely many neighbours, i.e., the graph is locally finite on $V$.
\end{remark}

We also need the definition of an optimal Hardy weight, confer \cite{KePiPo2} in the discrete $(p=2)$-setting and \cite{DP16} in the continuum.

\begin{definition}\label{def:optimal}
	The function $w\gneq 0$ is called an \emph{optimal ($p$-)Hardy weight} for the energy functional $h$ in $V$ if
	\begin{enumerate}[label=(\roman*)]
		\item $h-w_p$ is critical in $V$,
		\item\label{def:optimalNC} $h-w_p$ is null-critical with respect to $w$ in $V$,
		\item\label{def:optimalNI} $h-w_p\geq \lambda w_p$ fails to hold on $C_c(V\setminus K)$ for all $\lambda >0$ and finite $K\sse V$, i.e., $w$ is \emph{optimal near infinity} for $h$.
	\end{enumerate}
\end{definition}

The first main result tells us now  how to find such  optimal Hardy weights. The associated version in the continuum is given in \cite[Theorem~1.1]{Versano}. Recall that by the Harnack inequality, see \cite{F:GSR}, a positive superharmonic function is strictly positive.

\begin{theorem}[Optimal Hardy weights]\label{thm:mainresultc<0}
	Let $V\sse X$ be non-empty and infinite such that $(V,b|_{V\times V})$ is locally finite on $V$. Let $h$ be a subcritical $p$-energy functional in $V$ with arbitrary potential $c$, corresponding $p$-Schrödinger operator $H:=H_{b,c,p,m}$, and $p$-Laplacian $L:=L_{b,p,m}$, $p>1$. 
	
	Suppose that $0\lneq u \in \FF(V)\cap C(V)$ is a proper function of bounded oscillation on $V$ such that $\tilde{H} u \geq 0$ on $V$, where $\tilde{H}:=H_{b,C_{p}\cdot c,p,m}$, and
	$C_{p}:=\left(p/(p-1)\right)^{p-1}.$
	
	Furthermore, assume that $Lu\in \ell^1(V,m)$, $u\in\ell^{p-1}(V,c_-)$, and 
	\begin{enumerate}[label=(\alph*)]
		\item\label{thm:mainresultc<01} $u$ takes its maximum on $V$, or  there exists $S>0$ such that for all $x\in V$ with $u(x)>S$ we have $Lu(x) \leq  0$, and
		\item\label{thm:mainresultc<02}  $u$ takes its minimum on $V$, or  there exists $I>0$ such that for all $x\in V$ with $u(x)<I$ we have $Lu(x) \geq  0$.
	\end{enumerate}
	Then, \[w:=\frac{H(u^{(p-1)/p})}{u^{(p-1)^2/p}},\] multiplied with $m$ is an optimal $p$-Hardy weight of $h$ on $V$.
\end{theorem}
Note that a proper function of bounded oscillation on an infinite graph cannot have both, a maximum and a minimum. 

We remark that a sufficient condition for $Lu\in \ell^1(V,m)$ is $\tilde{H}u\in \ell^1(V,m)$ and $u\in \ell^{p-1}(V,\abs{c})$ by the triangle inequality. Moreover, if $\tilde{H}u \in C_c(V)$ -- in other words, $u$ is $p$-harmonic outside of a finite set with respect to $\tilde{H}$, then the conditions $\tilde{H}u \in \ell^1(X,m)$,   \ref{thm:mainresultc<01} and \ref{thm:mainresultc<02} are satisfied. In \cite{KePiPo2}, where the case $p=2$ was investigated, the assumption is that $u$ should be harmonic with respect to $H$. Hence, in a certain sense the assumptions here generalise the result in  \cite{KePiPo2}.  Furthermore, if $0\leq c\in C_c(V)$, then the conditions $u \in \ell^{p-1}(V,c)$ and  \ref{thm:mainresultc<02} are satisfied. 

In addition to that, we remark that in the linear ($p=2$)-case no multiplication of the potential with a constant $C_2$ is needed, see \cite{KePiPo2}. There, a discrete chain rule of the square root was applied which remains unknown for $p\neq 2$. However, also for the quasi-linear counterpart in the continuum (that is \cite[Theorem~1.1]{Versano}), the same constant is needed.

 As a corollary, we get optimal Hardy weights to Laplace-type operators. The corresponding version in the continuum is given in \cite[Theorem~1.5]{DP16}. 

\begin{corollary}[Optimal Hardy weights for non-negative potentials]\label{cor:mainresultc>0}
	Let $V\sse X$ be non-empty and infinite such that $(V,b|_{V\times V})$ is locally finite on $V$, $p>1$. Let $h$ be a subcritical $p$-energy functional in $V$ with non-negative potential $c$ and corresponding $p$-Schrödinger operator $H$, and $p$-Laplacian $L$. 
	Suppose that $0\lneq u \in \FF(V)\cap C(V)$ is a proper function of bounded oscillation on $V$ with $ Hu \geq 0$ on $V$ and $Lu\in \ell^1(V,m)$ such that
	\begin{enumerate}[label=(\alph*)]
		\item\label{thm:mainresultc>01} $u$ takes its maximum on $V$, or  there exists $S>0$ such that for all $x\in V$ with $u(x)>S$ we have $Lu(x) \leq  0$, and
		\item\label{thm:mainresultc>02}  $u$ takes its minimum on $V$, or  there exists $I>0$ such that for all $x\in V$ with $u(x)<I$ we have $Lu(x) \geq  0$.
	\end{enumerate}
	Then, \[w:=\frac{H(u^{(p-1)/p})}{u^{(p-1)^2/p}},\] multiplied with $m$ is an optimal $p$-Hardy weight of $h$ on $V$.
\end{corollary}
\begin{proof}
	Note that $C_p> 1$ in Theorem~\ref{thm:mainresultc<0}. Hence, $\tilde{H}u\geq Hu\geq 0$, and we can apply Theorem~\ref{thm:mainresultc<0}.
\end{proof}

Note that some explicit examples are given in Section~\ref{examples}.

\begin{remark}[Green's functions]
	It is shown in \cite{F:AAP} that whenever $h$ is subcritical in $X$ there exists a Green's function $g_o\colon X\to (0,\infty)$ for every $o\in X$ with the property that $Hg_o=1_o$. Thus, $g_o$ is a strictly positive superharmonic function that is harmonic outside of a singleton. Hence, Green's functions are natural candidates for the function $u$ in Theorem~\ref{thm:mainresultc<0} and Corollary~\ref{cor:mainresultc>0}. 
	
	Also note that there exist graphs associated with subcritical energy functionals whose corresponding Green’s functions do not vanish at every boundary point: Take, for instants, the standard line graph $\NN_0$ from the introduction with $c=0$ and $m=1$. On $\NN$, this graph is subcritical, and a positive minimal Green's function $G_1$ on $1$ is given by $G_1=1$ on $\NN$ and $G_1(0)=0$. Note that $G_1$ is still of bounded oscillation but not proper. Moreover, since $LG^{(p-1)/p}_1=1_1$ on $\NN$, the corresponding formula does not result in an optimal $p$-Hardy weight on $\NN$. This shows that the properness seems to be a natural requirement. Another example of a tree with this property is given in \cite[p. 240]{Woe-Book}. 
\end{remark}

\begin{remark}[Adding positive potentials]
	As a consequence of Lemma~\ref{lem:opti}, we want to highlight the following: Let $b$ be locally finite on $X$ and $W\gneq 0$. If $h$ admits an optimal Hardy weight $w$, and $h-w$ has a proper Agmon ground state of bounded oscillation, then $w+ W$ is an optimal Hardy weight of $h+W$.
	
	For example, in order to apply Corollary~\ref{cor:mainresultc>0}, it is sufficient to find a suitable function $u$ when $c=\epsilon 1_o$ for some $o\in X$ and $\epsilon >0$; or if $h_{b,0,p}$ is already subcritical, then one can even set $c=0$, which might simplify computations.
\end{remark}

The third main result tells us how large a Hardy weight might be in a neighbourhood of infinity. The counterpart in the continuum is \cite[Theorem~3.1 and Theorem~3.2]{KP20}, and \cite[Theorem~7.9]{HPR21}. In the following the potential can again be arbitrary.

We set $\deg_W(x):=\sum_{y\in W}b(x,y)$ for $W\sse X$ and $x\in X$.

\begin{theorem}[Decay of Hardy weights]\label{thm:decay}
	Let $V\sse X$ be connected and non-empty, and $h$ be non-negative on $C_c(V)$. Let $W\sse V$ be a non-empty set, and $K\sse V$ be a finite non-empty set.
	\begin{enumerate}[label=(\alph*)]
	\item\label{thm:decayCritical} Assume that $h$ is critical in $V$ with Agmon ground state $u\in C(V)\cap \ell^p(V\setminus W, \deg_W)$ such that $\sup_{x\in W, y\in V\setminus W, x\sim y}u(x)/u(y)< \infty$. 	
	Then for any Hardy weight $w$ on $V\setminus W$, we have \[w\in \ell^{1}(V\setminus W, u^{p}),\] i.e., $u\in \ell^{p}(V\setminus W, w)$.	
\item\label{thm:decaySubcritical} Assume that $h$ is subcritical in $V$, and that there are only finitely many edges between $V\setminus K$ and $K$. Let $v\in \MM(V\setminus K)$. Then for any Hardy weight $w$ on $V$, we have \[w\in \ell^{1}(V, v^{p}).\] 
	\end{enumerate}
\end{theorem}

The existence of a Hardy weight on $X\setminus W$ is ensured by Lemma~\ref{lem:exHW}. Moreover, the technical assumptions in the preamble of \ref{thm:decayCritical} are fulfilled if the subgraph $(V,b|_{V\times V})$ is locally finite on finite $W$. Also note that for any finite $K\sse X$, $\deg_K$ is a finite measure on $X$ by the local summability condition on $b$.

	The definition of optimality evolved historically, see \cite{DFP, DP16}. It is natural to ask whether there is a connection between optimality near infinity and null-criticality. In the continuum, it is shown in \cite[Corollary~3.4]{KP20}, that indeed, \ref{def:optimalNC} implies \ref{def:optimalNI}. Moreover, \cite[Remark~1.3]{DP16} states an example that the other implication (\ref{def:optimalNI} $\implies$ \ref{def:optimalNC}) fails in general.
	
	On graphs associated with linear Schrödinger operators with compactly supported potential part it is shown in \cite{KePiPo2}, that \ref{def:optimalNC} implies \ref{def:optimalNI} for a special Hardy weight.

Here, we will show in the following theorem, Theorem~\ref{thm:decayImpliesOpti}, that \ref{def:optimalNC} implies \ref{def:optimalNI} for all $p>1$, and all possible potentials if the graph is locally finite and the ground state is smooth enough. It can be interpreted as a corollary of Theorem~\ref{thm:decay}, and it includes the case in \cite{KePiPo2} also for $p= 2$. See Subsection~\ref{sec:ncopti} for a proof.

\begin{theorem}[Null-criticality implies optimality near infinity]\label{thm:decayImpliesOpti}
Let $V\sse X$ be connected and non-empty. Let $h-w_p$ be null-critical with respect to $w\gneq 0$ in $V$ with  Agmon ground state $u$ in $C(V)\cap \ell^p(V, b(x,\cdot))$ for all $x\in V$, and of bounded oscillation in $V$. Then $w$ is optimal near infinity for $h$.
\end{theorem}
Note that the technical assumptions minimise if the subgraph $(V,b|_{V\times V})$ is locally finite. Moreover, $C(V)\cap \ell^p(V, b(x,\cdot))= C(V)\cap F_{b,p+1}(V)\sse C(V)\cap F_{b,p}(V)$, see e.g. \cite[Lemma~4.1]{MS23} for the last inclusion.

\section{Decay of  Hardy Weights}\label{sec:decay}
Here we will prove Theorem~\ref{thm:decay} and Theorem~\ref{thm:decayImpliesOpti}. Both proofs use certain preliminary results. These are shown in the next two subsections. We start with the ground state representation formula which is also of fundamental importance in the proof of the other main results. Thereafter, we turn to convexity-type results. 

\subsection{Ground State Representation}
Here we recall a recently developed equivalence between the $p$-energy functional and the so-called simplified energies, see \cite{F:GSR} for details.

Let $p>1$ and $0\leq u\in \FF(V)$, $V\sse X$. The \emph{simplified energy (functional)} $h_{u}$ of $h$ with respect to $u$ on $C_c(V)$ is defined as
\begin{align*}
	h_{u}(\phi)&:=\sum_{x,y\in X}b(x,y) u(x)u(y)(\nabla_{x,y}\phi)^{2} \\
		&\qquad \cdot\left( \bigl(u(x)u(y)\bigr)^{1/2}\abs{\nabla_{x,y}\phi}+ \frac{\abs{\phi(x)}+ \abs{\phi(y)}}{2}\abs{\nabla_{x,y}u} \right)^{p-2},
\end{align*}
where we set $0\cdot \infty =0$ if $1<p<2$.

\begin{theorem}[Ground state representation, Theorem~3.1 in \cite{F:GSR}]\label{thm:GSR}
	Let $p> 1$, and $0\leq u\in \FF(V)$, $V\sse X$. Then, we have
	\begin{align}\label{eq:GSRI}
		 h(u\phi)- (muHu)_p(\phi)\asymp h_{u}(\phi), \qquad \phi\in C_c(V).
	\end{align}
	Furthermore, for $p=2$, the equivalence becomes an equality.
\end{theorem}

Of particular interest is the following special case of the ground state representation: Assume that $0\leq u\in \FF$ is strictly superharmonic, and set $w:= Hu/u^{p-1} \gneq 0$, then Theorem~\ref{thm:GSR} yields the following Hardy inequality,
\[		 h(\phi)\geq  w_p(\phi), \qquad \phi\in C_c(X).\]
In the general form of Theorem~\ref{thm:GSR}, the non-negativity of the left-hand side is also known as Picone's inequality, see \cite{F:AAP} for details.

Moreover, if $u\in \FF$ is harmonic in $X$, then we deduce from Theorem~\ref{thm:GSR} a second very useful equivalence:
\[ h(u\phi)\asymp h_{u}(\phi), \qquad \phi\in C_c(X).\]

From the inequalities in Theorem~\ref{thm:GSR}, we get as a consequence estimates between the energy associated with the Schrödinger operator and other functionals, which are usually also referred to as \emph{simplified energies} (see e.g. \cite{DP16, PTT08}). They all are called simplified, because they consist of non-negative terms only, and the difference operator $\nabla$ applies either to $u$ or $\phi$ but not to the product $u\cdot \phi$.

We set on $C_c(X)$
\begin{align*}
	h_{u,1}(\phi):= \sum_{x,y\in X}b(x,y) (u(x)u(y))^{p/2}\abs{\nabla_{x,y}\phi}^{p},
\end{align*}
and for $p\geq 2$, we define on $C_c(X)$
\[h_{u,2}(\phi):=\sum_{x,y\in X}b(x,y)u(x)u(y)\abs{\nabla_{x,y}u}^{p-2} \left(\frac{\abs{\phi(x)}+ \abs{\phi(y)}}{2}\right)^{p-2}\abs{\nabla_{x,y}\phi}^{2}.\]

\begin{corollary}[Corollary~3.2 in \cite{F:GSR}]\label{cor:GSR}
	If $1<p\leq 2$, then there is a positive constant $C_{p}$ such that for all $0\leq u\in \FF$
	\begin{align}\label{eq:GSRI_p<2}
		h(u\phi) - (muHu)_p(\phi)\leq C_p h_{u,1}(\phi),\quad \phi\in C_c(X).
	\end{align}
	and if $p\geq 2$, the reversed inequality in \eqref{eq:GSRI_p<2} holds true, i.e., 
	\begin{align}\label{eq:GSRI_p>2}
		h(u\phi) - (muHu)_p(\phi)\geq C_p h_{u,1}(\phi),\quad \phi\in C_c(X).
		\end{align}
 Furthermore, both inequalities become equalities if $p=2$.
	
	Moreover, if $p\geq 2$, we have for all $0\leq u\in \FF$ and $\phi\in C_c(X)$
	\begin{align}\label{eq:GSRI_p>2Triangle}
	h(u\phi)- (muHu)_p(\phi) \asymp h_{u,1}(\phi)+h_{u,2}(\phi), \qquad \phi\in C_c(X).
		\end{align}
\end{corollary}

The statements above do not include the case $p=1$. This is because a quantification of the strict convexity of the mapping $x\mapsto \abs{x}^p$, $p> 1$, is used in the proofs.

To get optimal Hardy weights, we will use the following observation frequently: Set on $C_c(X)$
\[h_{u,3}(\phi):=\sum_{x,y\in X}b(x,y) \left(\frac{\abs{\phi(x)}+ \abs{\phi(y)}}{2}\right)^{p}\abs{\nabla_{x,y}u}^{p}.\]
Then, applying Theorem~\ref{thm:GSR}, Corollary~\ref{cor:GSR}, and Hölder's inequality with $p'=p/2$ and $q'=p/(p-2)$, we get the existence of a positive constant $C_p$ such that for all $0\leq u\in \FF$,
\begin{align}\label{eq:GSRHoelder}
	h_{u}(\phi)\leq C_p\cdot \begin{cases}h_{u,1}(\phi), &\qquad 1< p\leq 2, \\ h_{u,1}(\phi) + \bigl(\frac{h_{u,1}(\phi)}{h_{u,3}(\phi)}\bigr)^{2/p}h_{u,3}(\phi), &\qquad p>2.\end{cases}
\end{align}

Moreover, we will often employ some characterisations of criticality from \cite{F:GSR,F:AAP}, which are stated next.

A sequence $(e_n)$ in $C_c(V)$, $V\sse X$, of non-negative functions is called \emph{null-sequence} in $V$ if there exists $o\in V$ and $\alpha>0$ such that $e_n(o)=\alpha$ and $h(e_n)\to 0$.

\begin{theorem}
\label{thm:charCriti}
Let $h$ be non-negative on $C_c(V)$, where $V\sse X$ is connected and non-empty. Then the following statements are equivalent:
\begin{enumerate}[label=(\roman*)]
		\item\label{thm:critical1} $h$ is critical in $V$.
		\item\label{thm:critical2} For any $o\in V$ and $\alpha>0$ there is a null-sequence $(e_n)$ in $V$ such that $e_n(o)=\alpha$, $n\in\NN$.
		\item\label{thm:critical4Neu} For any positive superharmonic function $u\in \FF(V)$ in $V$ and any null-sequence $(e_n)$ in $V$ there exists a positive constant $C$ such that 
		$e_n(x)\to C\, u(x)$ for all $x\in V$ as $n\to \infty$.
		\item\label{thm:critical22} There exists a unique positive superharmonic function in $V$ (up to linear dependence) and this function is strictly positive and harmonic in $V$, the so-called Agmon ground state.
	\end{enumerate}
\end{theorem}
\begin{proof}
Combine \cite[Theorem~2.6]{F:AAP} with \cite[Theorem~4.1]{F:GSR} and the Agmon-Allegretto-Piepenbrink theorem (\cite[Theorem~2.3]{F:AAP}). This shows in fact the statement for $V=X$. However, it is not difficult to redo the proof for connected non-empty $V\subsetneq X$.
\end{proof}

Furthermore, we will need that the null-sequence in Theorem~\ref{thm:charCriti} can be chosen to be \emph{increasing}. This is shown in Appendix~\ref{app}. 

The following result is proven in \cite[Corollary~5.6]{F:AAP}.

\begin{lemma}\label{lem:w>0}
	Let $V\sse X$ be connected. If $h$ is subcritical in $V$, then there is a strictly positive Hardy weight $w$ on $V$.
\end{lemma}

\subsection{Convexity-type Results}
First, we show a Poincar\'{e}-type inequality. This is the discrete counterpart of \cite[Theorem~1.6.4]{PT07}, see also \cite{HPR21, PR15, PP}.

\begin{lemma}[Poincar\'{e} inequality]\label{lem:PT07_164}
	Let $V\sse X$ be connected and non-empty. Assume that $h$ is critical in $V$ with Agmon ground state $u$. Then, there exists a strictly positive function $w$ on $V$ such that for every $\psi\in C_c(V)$ with $\ip{u}{\psi}_{V}\neq 0$ there exists a positive constant $C=C(\psi)$ such that 
	\[C^{-1}w_p(\phi)\leq h(\phi)+ C\abs{\ip{\phi}{\psi}_{V}}^{p}, \qquad \phi\in C_c(V).\]
\end{lemma}
\begin{proof}
	We divide the proof into two steps.
	
\emph{1. Claim: For all non-empty $W\sse V$ there exists a strictly positive function $w\in C(V)$ such that 
	\[w_p(\phi)\leq h(\phi)+ (m1_W)_p(\phi), \qquad \phi\in C_c(V).\]}
	Note that the right-hand side is a  $p$-energy functional which we denote by $\hh$. Then, clearly, $\hh$ is non-negative on $C_c(V)$. Moreover, $\hh \geq (m1_W)_p$ on $C_c(V)$, i.e., $\hh$ is subcritical (with $p$-Hardy weight $m1_W$). Since $\hh$ is subcritical in $V$ such a $w$ exists by Lemma~\ref{lem:w>0}. 	
	
\emph{2. Claim: Let $u$ be an Agmon ground state of the critical $p$-energy functional $h$ on $V$. For every $\psi\in C_c(V)$ with $\ip{u}{\psi}_{X}\neq 0$,  there exists a positive constant $C=C(\psi)$ such that for all finite and non-empty $K\sse V$, we have
	\[(m1_K)_p(\phi)\leq C\bigl( h(\phi)+ \abs{\ip{\phi}{\psi}_{V}}^{p} \bigr), \qquad \phi\in C_c(V).\]}
Assume that the inequality above is wrong. Then there is a sequence $(\phi_n)$ in $C_c(V)$ such that $h(\phi_n)\to 0$, $\ip{\phi_n}{\psi}_{V}\to 0$, but $(m1_K)_p(\phi_n)\to \alpha\in (0,\infty]$. Thus, since $K$ is finite and non-empty, there exists $o\in K$ and $\epsilon> 0$ such that for all $n_k\geq n_0$, $\phi_{n_k}(o)\geq  \epsilon$. Set $e_k=\phi_{n_k}/\phi_{n_k}(o)$ for all $n_k\geq n_0$. Then, $(e_k)$ is a null sequence of $h$ in $V$. Thus, $e_k\to \tilde{C}\, u$ by Theorem~\ref{thm:charCriti} for some positive constant $\tilde{C}$. Since $\ip{e_k}{\psi}_{V}\to \tilde{C}\ip{u}{\psi}_{V}\neq 0$, we have a contradiction, and the second claim is proven.

Now, applying both claims yields the result.
\end{proof}

We can prove the following convexity-type result. The counterpart in the continuum can be found in \cite[Proposition~4.3]{PT07}.
\begin{proposition}\label{prop:convexh}
	Let $c_0, c_1$ be two potentials such that $c_0\neq c_1$ on a connected and non-empty subset $V\sse X$. Denote the corresponding $p$-energy functionals by $h_0,h_1$, respectively. Furthermore, define for all $t\in [0,1]$
	\[h_t(\phi):= t h_1(\phi)+(1-t)h_0(\phi), \qquad \phi\in C_c(X).\]
	Let $h_i$ be non-negative on $C_c(V)$, $i=0,1$. Then $h_t$ is non-negative on $C_c(V)$ for all $t\in [0,1]$. Moreover, $h_t$ is subcritical in $V$ for all $t\in (0,1)$.
\end{proposition}
\begin{proof}
	The proof follows the ideas from its counterpart in the continuum. Here are the details: Clearly, $h_{t}$ is non-negative on $C_c(V)$ for all $t\in [0,1]$. 
	
	If $h_0$ is subcritical in $V$ with Hardy weight $w_0$, then, $h_t$, $t\in [0,1)$, is subcritical in $V$ with Hardy weight $(1-t)w_0$. Analogously, we can argue if $h_1$ is subcritical in $V$. 
	
	It remains the case that both $h_0$ and $h_1$ are critical in $V$. Denote by $u_0$ and $u_1$ the corresponding Agmon ground states normalised at some $o\in V$. Assume that for some $t\in (0,1)$ also $h_t$ is critical. Then by Theorem~\ref{thm:charCriti}, $h_t$ has an Agmon ground state $u_t$ on $V$ normalised at $o\in V$ and a null-sequence $(e_n)$ on $V$, such that $0\leq e_n\to u_t$ pointwise on $V$. We show that this results in a contradiction. 
	
	We claim that $c_0\neq c_1$ implies that $u_t$ cannot be a multiple of $u_1$ or $u_2$. Indeed, let $H_t$ be the $p$-Schrödinger operator associated to $h_t$, $t\in [0,1]$, and assume that $u_t=C u_0$ for some $C> 0$. Since $u_t$ is an Agmon ground state for $t\in [0,1]$, we have $0=H_tu_t=tH_1u_t+(1-t)H_0u_t= tC^{p-1}H_1u_0$ on $V$. Hence, $u_0$ is also a global $p$-harmonic function for $H_1$ on $V$. Since $u_0(o)=1=u_1(o)$, we have by the uniqueness of the Agmon ground state (up to multiplication by a constant) that $u_0=u_1$. But this yields, $c_1u_1^{p-1}=-mLu_1=-mLu_0=c_0u_0^{p-1}=c_0u_1^{p-1}$, i.e., $c_0=c_1$ on $V$. Interchanging the role of $u_1$ and $u_0$ yields the claim.
	
	Let $(e_n^t)$ be a null sequence for $h_t$ converging pointwise to $u_t$, $t\in [0,1]$. Then $h_t(e_n^t)\to 0$ implies $h_1(e_n^t)\to 0$ and $h_0(e_n^t)\to 0$. Hence, $(e_n^t)$ is a null sequence for both $h_1$ and $h_0$. By construction, this sequence converges pointwise to $u_1$ and $u_0$, and also $u_t$. But since $c_0\neq c_1$, we have that $u_t$ cannot be a multiple of $u_1$ or $u_2$, and get a contradiction. Therefore, $h_t$ does not admit an Agmon ground state and by Theorem~\ref{thm:charCriti}, it is subcritical in $V$ for $t\in (0,1)$.
\end{proof}

A consequence of the previous proposition is given next. For the counterpart in the continuum confer \cite[Proposition~4.4]{PT07} or \cite[Proposition~4.19]{PP}.
\begin{corollary}\label{cor:PT07Prop44}
	Let $V\sse X$ be connected and non-empty, and let $h$ be subcritical in $V$. Let $\ww\in C_c(V)$ such that $\ww(x)<0$ for some $x\in V$. Then there exists $\tau_+ \in (0,\infty)$ and $\tau_-\in [-\infty, 0)$ such that $h+t\cdot \ww_p$ is subcritical in $V$ for $t\in (\tau_-,\tau_+)$, and such that $h+\tau_+\cdot \ww_p$ is critical in $V$.
\end{corollary}
\begin{proof}
	By Lemma~\ref{lem:w>0}, there is a strictly positive Hardy weight $w$ associated with $h$ on $V$. Thus, $(w+t\cdot \ww)_p(\phi)\leq h(\phi)+t \ww_p(\phi)$ for all $\phi\in C_c(V)$. Since $w>0$ on $V$, and $\ww$ is of compact support, we get that $w+t\cdot \ww> 0$ on $V$ for sufficiently small values of $\abs{t}$. Let
	\[S:=\set{t\in \RR: h+t\cdot \ww_p \text{ is subcritcal in } V}.\]
	By the previous proposition, Proposition~\ref{prop:convexh}, $S$ is an interval. Let denote by $\tau_+$ and $\tau_-$, the right and left boundary point of this interval, respectively. 
	
	We show that $\tau_+< +\infty$: Since there is $x\in V$ such that $\ww(x)<0$, we have $\ww_p(1_x)<0$ and thus, $h(1_x)+t\cdot \ww_p(1_x)<0$ for sufficiently large $t$.
	
	We show that $\tau_-=-\infty$ can be obtained: Assume that $\ww \lneq 0$ on $V$, then $\ww_p(\phi)\leq 0$ for all $\phi\in C_c(V)$, and thus for all $t<0$, we get $t\in S$.
	
	We show that $h+\tau_+ \ww_p$ is critical: Clearly, $h+\tau_+ \ww_p$ is non-negative on $C_c(V)$. Assume that $h+\tau_+ \ww_p$ is subcritical, i.e., $\tau_+\in S$. Arguing as in the beginning, we see that for sufficiently small $\epsilon >0$, we get that $h+(\tau_++\epsilon) \ww_p$ is subcritical. But this contradicts that $\tau_+$ is the right boundary point of the interval $S$. 
\end{proof}

The following result can be seen as the counterpart of Corollary~\ref{cor:PT07Prop44} for critical $p$-energy functionals. Confer \cite{PT07, PP, PR15} for the local case. 

\begin{corollary}\label{cor:PT07Prop45}
	Let $V\sse X$ be connected and non-empty, and let $h$ be critical in $V$ with Agmon ground state $u\in C(V)$.
	\begin{enumerate}[label=(\alph*)]
	\item\label{cor:PT07Prop45A} Assume that $\ww \in \ell^{\infty}(V)$ and there exists $\tau_+\in (0,\infty]$ such that $h+\tau_+\cdot \ww_p$ is non-negative in $C_c(V)$. Then, for any null-sequence $(e_n)$ of $h$ in $V$, we have $\liminf_{n\to\infty}\ww_p(e_n)>0$.
	\item\label{cor:PT07Prop45B} Assume that $\ww \in C_c(V)$ and $\ww_p(u)> 0$. Then there exists $\tau_+\in (0,\infty]$ such that $h+t\cdot \ww_p$ is subcritical for all $t\in (0,\tau_+)$.
	\end{enumerate}
\end{corollary}
\begin{proof}
	Ad \ref{cor:PT07Prop45A}: By Proposition~\ref{prop:convexh}, $h+ t\cdot \ww_p$ is subcritical for all $t\in (0,\tau_+)$. Thus, we can use Lemma~\ref{lem:w>0}, and get the existence of a strictly positive $p$-Hardy weight $w$ associated with the $p$-energy functional $h+ t\cdot \ww_p$. Let $(e_n)$ be an arbitrary null-sequence of $h$ such that  $e_n\to C\cdot u> 0$ for some positive constant $C$, which exists by Theorem~\ref{thm:charCriti}. By Fatou's lemma, we infer
	\begin{align*}
		t\cdot \liminf_{n\to\infty} \ww_p(e_n)&= \liminf_{n\to\infty} h(e_n) + \liminf_{n\to\infty}t\cdot \ww_p(e_n) \\
		&\geq \liminf_{n\to\infty}w_p(e_n) \geq C^p\cdot w_p(u)> 0.
	\end{align*}
	Dividing by $t>0$ yields \ref{cor:PT07Prop45A}.
	
	Ad \ref{cor:PT07Prop45B}: The strategy is as follows: If there is $\tau_+\in (0,\infty]$ such that $h+\tau_+\cdot \ww_p$ is non-negative in $C_c(X)$, then by Proposition~\ref{prop:convexh}, we get that $h+t\cdot \ww_p$ is subcritical for all $t\in (0,\tau_+)$.
	
	Thus, let us assume that such a $\tau_+\in (0,\infty]$ does not exists. We will show that this leads to a contradiction. Hence, we assume that for all $t>0$, there exists $\phi_t\in C_c(V)$ such that
	\begin{align}\label{eq:PT07_1}
	h(\phi_t)+t\cdot \ww_p(\phi_t)<0.
	\end{align}
	By the reversed triangle inequality, $h(\phi_t)\geq h(\abs{\phi_t})$. Thus, we can assume without loss of generality that $\phi_t\geq 0$. Since $h$ is non-negative on $C_c(V)$, it follows from \eqref{eq:PT07_1}, that
	\begin{align}\label{eq:PT07_2}
	\ww_p(\phi_t)< 0.
	\end{align}
	In particular, by Inequality~\eqref{eq:PT07_2}, we have $\supp(\phi_t)\cap \supp (\ww)\neq \emptyset$. 
	
	Let $\psi_t:=\phi_t/\norm{\phi_t}_{p,1_{\supp \ww}}$. Then, using that $\ww$ is bounded on $V$, we get
	\[\lim_{t\to 0}t\cdot \ww_p(\psi_t)\leq \sup_{V}(\ww) \cdot \lim_{t\to 0} t\cdot \norm{\psi_t}^{p}_{p,1_{\supp \ww}}=0.\]
	Since $h$ is non-negative on $C_c(V)$ we obtain using \eqref{eq:PT07_1} and the calculation above,
	\[0\leq \lim_{t\to 0}h(\psi_t)\leq  \lim_{t\to 0}t\cdot \ww_p(\psi_t) =0.\]
	Thus, $h(\psi_t)\to 0$ as $t\to 0$. 
	
	Since $\ww$ is finitely supported and since $\supp(\phi_t)\cap \supp (\ww)\neq \emptyset$ for all $t>0$, we have that for some $o\in \supp (\ww)$, there exists $\epsilon>0$ and a subsequence $(\psi_{t_n})$ of the net $(\psi_t)$ such that $\psi_{t_n}(o)\geq\epsilon$, and $t_n\to 0$ as $n\to\infty$.
	
	Let $\psi_n:=\psi_{t_n}/\psi_{t_n}(o)$, then $h(\psi_n)=(1/\psi_{t_n}(o))^ph(\psi_{t_n})\to 0$ as $n\to \infty$. Hence, $(\psi_n)$ is a null sequence	of $h$, and thus, $\psi_n\to C u$ for some positive constant $C$. By the assumptions on $\ww$, we thus have
	\[\lim_{n\to \infty} \ww_p(\psi_n)=C^p\ww_p(u)> 0.\]
	This contradicts \eqref{eq:PT07_2}.
\end{proof}

We continue with a simple observation which is basically a consequence of Lemma~\ref{lem:PT07_164}. Confer \cite[Proposition~4.2]{PT07}, or \cite[Proposition~4.18]{PP}  for the counterpart in the continuum. 

Using a different and more technical method, it is shown in \cite[Corollary~4.2]{F:AAP} that if $h$ is non-negative on $C_c(V\cup\set{o})$ for some $V\subsetneq X$ and $o\in \partial V$ then $h$ is subcritical in $V$. This clearly makes following lemma, Lemma~\ref{lem:exHW}, redundant. However, for convenience, we give a short argument based on the Poincaré inequality, Lemma~\ref{lem:PT07_164}, here.
\begin{lemma}\label{lem:exHW}
	Let $W\subsetneq V\sse  X$ be both connected and non-empty. 
	\begin{enumerate}[label=(\alph*)]
		\item\label{lem:exHW1} If $h$ is non-negative on $C_c(V)$, then $h$ is subcritical in $W$.
		\item\label{lem:exHW2} If $h$ is critical in $W$, then $h$ is supercritical in $V$.
	\end{enumerate}
\end{lemma}
\begin{proof}
	If $h$ is subcritical in $V$, then Lemma~\ref{lem:w>0} implies that we have a Hardy weight $w$ which is strictly positive in $V$. Thus, taking $w\cdot 1_{W}$, we have a Hardy weight on $W$, and  $h$ is subcritical in $W$.
	
	Let $h$ be critical in $V$ with Agmon ground state $u\in C(V)$. Take $o\in V\setminus W$. By  Lemma~\ref{lem:PT07_164}, there exists a strictly positive function $w$ on $V$ and a positive constant $C$ such that 
	\[C^{-1}w_p(\phi)\leq h(\phi)+ C\abs{\ip{\phi}{1_o}_{V}}^{p}= h(\phi), \qquad \phi\in C_c(W).\]
	Thus, $h$ is subcritical in $W$. This proves \ref{lem:exHW1}. The second statement \ref{lem:exHW2} is just the contraposition of the first one if $h$ is non-negative on $C_c(W)$.
\end{proof}

\subsection{Proof of Theorem~\ref{thm:decay}}
Recall that $\deg_V(x):=\sum_{y\in V}b(x,y)$ for $V\sse X$ and $x\in X$.
\begin{proof}[Proof of Theorem~\ref{thm:decay}]
Let us mention that there is nothing to prove for $W=V$ or $K=V$. Hence, in the following, we assume that $W\subsetneq V$ or $K\subsetneq V$.
	
Ad~\ref{thm:decayCritical}: Let $h$ be critical in $V$. Then, by Theorem~\ref{thm:charCriti}, there is a null sequence $(e_n)$ in $V$ such that $e_n\to u$ pointwise, where $u\in C(V)$ is the Agmon ground state of $h$ in $V$.  
	By Lemma~\ref{lem:exHW}, we have a $p$-Hardy weight $w$ on $V\setminus W$, i.e.,
	\[ w_p(\phi)\leq h(\phi), \qquad \phi\in C_c(V\setminus W).\]
	 By Fatou's lemma, we have 
	\[w_p(u1_{V\setminus W})\leq \liminf_{n\to\infty} w_p(e_n1_{V\setminus W}).\]
	Thus, we have to show that the right-hand side is finite.	We calculate, using Theorem~\ref{thm:GSR},
	\begin{align*}
		 w_p(e_n1_{V\setminus W})\leq  h(e_n1_{V\setminus W})\asymp  h_u(e_n1_{V\setminus W}/u).
	\end{align*}
Moreover, 
\begin{align*}
h_u(e_n1_{V\setminus W}/u)
&=\sum_{x,y\in V\setminus W}b(x,y)u(x)u(y)\abs{\nabla_{x,y}\frac{e_n}{u}}^{2} \\
&\qquad\cdot \left((u(x)u(y))^{1/2}\abs{\nabla_{x,y}\frac{e_n}{u}}+ \frac{\abs{\frac{e_n(x)}{u(x)}}+\abs{\frac{e_n(y)}{u(y)}}}{2}\abs{\nabla_{x,y}u}\right)^{p-2} \\ 
&\quad+ 2\cdot \sum_{x\in W, y\in V\setminus W}b(x,y)u(x)u(y)\frac{e^2_n(y)}{u^2(y)}\\
&\qquad\cdot \left( \bigl(u(x)u(y)\bigr)^{1/2}\cdot \frac{e_n(y)}{u(y)}+ \frac{1}{2}\cdot \frac{e_n(y)}{u(y)}\abs{\nabla_{x,y}u} \right)^{p-2} 
\end{align*}	
This first sum on the right-hand side can be estimated from above by	$h_u(e_n/u)$. Since $h_u(e_n/u)\asymp h(e_n)\to 0$ as $n\to \infty$, we only need to discuss the second sum. Here we have, using that $u$ satisfies $\sup_{x\in W, y\in V\setminus W, x\sim y}u(x)/u(y)< \infty$, the existence of a constant $C_p>0$ such that
\begin{align*}
&\sum_{x\in W, y\in V\setminus W}b(x,y) u(x)u(y)\frac{e^2_n(y)}{u^2(y)}\left( \bigl(u(x)u(y)\bigr)^{1/2}\cdot \frac{e_n(y)}{u(y)}+ \frac{1}{2}\cdot \frac{e_n(y)}{u(y)}\abs{\nabla_{x,y}u} \right)^{p-2} \\
&=\sum_{x\in W, y\in V\setminus W}b(x,y)\frac{u(x)}{u(y)}\left( \left(\frac{u(x)}{u(y)}\right)^{1/2}+\abs{\frac{u(x)}{u(y)}-1}\right)^{p-2}e^p_n(y)\\
&\leq C_p \sum_{x\in W, y\in V\setminus W}b(x,y)e^p_n(y)= C_p \sum_{y\in V\setminus W}e^p_n(y)\deg_{W}(y).
\end{align*}
 
We can assume by the result in Appendix~\ref{app} that $e_n\leq u$. Since we have by assumption $u\in \ell^p(V\setminus W, \deg_W)$, we proved \ref{thm:decayCritical}.

Ad~\ref{thm:decaySubcritical}:	Let $h$ be subcritical in $V$ with corresponding Hardy weight $w$, i.e., $h-w_p\geq 0$ on $C_c(V)$. Let $v\in \MM(V\setminus K)$ for some finite $K\subsetneq V$.
	
	Using Corollary~\ref{cor:PT07Prop44}, there is a potential $\ww\in C_c(K)=C(K)$ such that $h-\ww_p$ is critical in $V$. Let denote the Agmon ground state on $V$ of $h-\ww_p$ by $u$. Hence, $u\in \MM(V)\sse \MM(V\setminus K)$ with respect to $h$. Furthermore, since $u$ and $v$ are minimal on $V\setminus K$ and on the finite set $K$ there always exists a constants $C>0$ for the strictly positive functions $u$ and $v$ such that $u\leq C v$, we have $v \asymp u$ on $V\setminus K$.
	
Since for all $\phi\in C_c(V\setminus K)$
	\[(h-\ww_p)(\phi)=h(\phi)\geq w_p(\phi)= w_p(1_{V\setminus K}\phi),\]
	$w$ is also a Hardy weight for $h-\ww_p$ on $C_c(V\setminus K)$. Thus, we can use the first part, and get that $w\in \ell^{1}(V\setminus K, v^{p})=\ell^{1}(V\setminus K, u^{p})=\ell^1(V,u^p)$.
\end{proof}

\subsection{Proof of Theorem~\ref{thm:decayImpliesOpti}}\label{sec:ncopti}
Here we show, in particular, that if the graph is locally finite and the Agmon ground state is proper and of bounded oscillation, then null-criticality implies optimality near infinity.

\begin{proof}[Proof of Theorem~\ref{thm:decayImpliesOpti}]
	 Let $u$ be an Agmon ground state of the critical energy functional $h-w_p$ in $V$. We show the contraposition: Assume that $w$ is not optimal near infinity in $V$, then there exists $\mu> 1$ and $K\sse V$ finite, such that $h-\mu\cdot w_p \geq 0$ on $C_c(V\setminus K)$. Thus, $(h-w_p)- (\mu -1)w_p\geq 0$ on $C_c(V\setminus K)$. By the assumptions, we can use Theorem~\ref{thm:decay}~\ref{thm:decayCritical}, and get $w\in \ell^{1}(V, u^{p})$. The latter is equivalent to $u\in \ell^{p}(V,w)$, and therefore, $w$ is not null-critical.
\end{proof}

\section{Optimal Hardy Weights}\label{sec:opti}
In the following subsections, we will derive optimal Hardy weights for energy functionals with arbitrary potential part. This will be achieved by the virtue of a coarea formula, Proposition~\ref{prop:coarea} in Subsection~\ref{sec:coarea}. Moreover, we will need the technicality that roots of superharmonic functions are also superharmonic. This is discussed in Subsection~\ref{sec:supersolution}. We start proving optimality by showing criticality in Subsection~\ref{sec:criti}, and show null-criticality and optimality near infinity in the two succeeding subsections. The proofs of the criticality and null-criticality work mostly along the lines of the proofs of \cite{KePiPo2} and \cite{DP16, Versano}, by either generalising from $p=2$ to $p\in (1,\infty)$ or by using discrete non-local versions of local methods in \cite{DP16, Versano}. The proof of the optimality near infinity is instead inspired by results in \cite{KP20}, confer also Theorem~\ref{thm:decay} and Theorem~\ref{thm:decayImpliesOpti}.

But before we start, we note the following discrete version of \cite[Lemma~2.35]{Versano}.
\begin{lemma}[Versano's lemma]\label{lem:opti}
	Let $V\sse X$ be connected and non-empty. Assume that $h$ admits an optimal $p$-Hardy weight $w$ on $V$, and $h-w_p$ has an Agmon ground state  in $C(V)\cap \ell^p(V, b(x,\cdot))$ for all $x\in V$, which is of bounded oscillation in $V$. Let $\omega \in C(X)$ be such that $\omega \geq -\epsilon \cdot w$ on $V$ for some $\epsilon \in [0,1)$. Then, $w+\omega$ is an optimal $p$-Hardy weight of $h+\omega_p$ in $V$.
\end{lemma}
\begin{proof}
	Firstly, $h+\omega_p$ is non-negative on $C_c(V)$ because of $\omega \geq -\epsilon w$ on $V$, and $w\gneq 0$ is a $p$-Hardy weight of $h$ in $V$, i.e., $h- \omega_p\geq h-\epsilon w_p\geq h-w_p\geq 0$ on $C_c(V)$. Since $(h+\omega_p)-(w+\omega)_p= h-w_p$ on $C_c(V)$, also the right-hand side is critical in $V$ with the same ground state $u$. Moreover, because of  $(w+\omega)_p(u)\geq (1-\epsilon)w_p(u)=\infty$, the functional $(h+\omega_p)-(w+\omega)_p$ is null-critical with respect to $w+\omega$. Thus, by Theorem~\ref{thm:decayImpliesOpti}, $w+\omega \gneq 0$ is an optimal $p$-Hardy weight of $h+\omega_p$ in $V$.
\end{proof}

Morally, this lemma tells us that sign changing and non-positive potentials are of particular interest. And once, we found a nice optimal Hardy weight for some energy functional, we get a family of optimal Hardy weights by simply adding certain potentials at the weight and the functional.
\subsection{Supersolution Constructions}\label{sec:supersolution}
One particular crucial step in the proof of Theorem~\ref{thm:mainresultc<0} will be to show that for a given positive and superharmonic function, certain roots of this function are again superharmonic. It turns out that for positive potentials, this can be deduced from standard techniques, see Proposition~\ref{prop:NewSuperharmonicFromOldc>0}. However, for arbitrary potentials a slightly modified Schrödinger operators has to be considered. This is shown next. In this subsection, we also consider the case $p= 1$. 

The following result can be seen as the discrete counterpart to \cite[Lemma~2.10]{DP16} and \cite[Corollary~3.6]{Versano} and uses the mean value theorem to circumvent the use of the chain rule, confer also with \cite[p. 350]{HuK} and \cite[Section~2.3]{KLW21}.

\begin{proposition}[Supersolutions via the mean value theorem]\label{prop:MeanValue}
	Let $p\geq 1$ and $u\in \FF(\set{x})$ for some $x\in X$. Let $\phi\in C^2([\inf_{\set{x}\cup \partial \set{x}} u,\sup_{\set{x}\cup \partial \set{x}}  u],\RR)$ be an increasing and concave function, i.e., $\phi', -\phi'' \geq 0$. 
	Then, we have
	\begin{align}\label{eq:MeanValue}
	L (\phi\circ u) (x) \geq ((\phi'\circ u)(x))^{p-1} Lu(x).
	\end{align}
	In particular, if $u(x)\neq u(y)$ for some $y\sim x$ and $\phi$ is strictly increasing, then the inequality in \eqref{eq:MeanValue} is strict.
	
	Moreover, if either
	\begin{enumerate}[label=(\alph*)]
		\item\label{prop:MeanWurzel}
		$u(x)>0$, and we have that $\phi(a)\geq \alpha \cdot a \cdot  \phi'(a)$ for some $\alpha >0$ and all $a\in [\inf_{\set{x}\cup \partial \set{x}}  u,\sup_{\set{x}\cup \partial \set{x}}  u]$, and also $H_{b,\alpha^{p-1}\cdot c,p,m}u(x) \geq 0$, or
		\item\label{prop:MeanSgn} $Hu(x) \geq 0$, and we have $c(x)=0$, or $\p{(\phi\circ u)(x)}= ((\phi'\circ u)(x))^{p-1}$, or $\sgn c(x)= \sgn (\p{(\phi\circ u)(x)}- ((\phi'\circ u)(x))^{p-1})$.
	\end{enumerate}				
	Then, $H(\phi\circ u)(x) \geq 0$.
	
	In particular, if $u(x)>0$, and also $H_{b,q^{p-1}\cdot c,p,m}u(x) \geq 0$ for some $q\geq 1$, then $H( u^{1/q})(x) \geq 0$. If, in addition to that, $u(x)\neq u(y)$ for some $y\sim x$, then we have $H( u^{1/q})(x) > 0$.
\end{proposition}
\begin{proof}
	By the mean value theorem, for all $z,y \in X$, there exists $\xi\in [u(z)\wedge u(y), u(z)\vee u(y)]$ such that
	\[ \nabla_{z,y} (\phi\circ u) = \phi'(\xi)\nabla_{z,y}u.\]
	Thus, for this fixed $x\in X$,
	\[m(x)L(\phi\circ u)(x)= \sum_{y\in X}b(x,y) \p{\phi'(\xi)\nabla_{x,y}u}. \]
	Since $\phi' \geq 0$, we obtain
	\begin{align*}
	\ldots = \sum_{y: \nabla_{x,y}u> 0}b(x,y) (\phi'(\xi))^{p-1}(\nabla_{x,y}u)^{p-1}-\sum_{y: \nabla_{x,y}u< 0}b(x,y) (\phi'(\xi))^{p-1}(\nabla_{y,x}u)^{p-1}.
	\end{align*}
	Moreover, because of $-\phi''\geq 0$, we can estimate as follows
	\begin{multline*}
	\ldots \geq \sum_{y: \nabla_{x,y}u> 0}b(x,y) (\phi'(u(x)))^{p-1}(\nabla_{x,y}u)^{p-1} \\
	-\sum_{y: \nabla_{x,y}u< 0}b(x,y) (\phi'(u(x)))^{p-1}(\nabla_{y,x}u)^{p-1}\\
	= m(x)(\phi'(u(x)))^{p-1}Lu(x).
	\end{multline*}
	This shows \eqref{eq:MeanValue}.
	
	Ad~\ref{prop:MeanWurzel}: Using the extra assumption on $\phi$ and $u$, and \eqref{eq:MeanValue}, we get for all $x\in V$
	\begin{align*}
	H(\phi\circ u)(x)&\geq ((\phi'\circ u)(x))^{p-1} \left( Lu(x)+ \frac{c(x)}{m(x)}\cdot \left( \frac{(\phi \circ u)(x)}{(\phi'\circ u) (x)}\right)^{p-1} \right)\\
	&\geq ((\phi'\circ u)(x))^{p-1} \left( Lu(x)+ \frac{c(x)}{m(x)}\cdot \left( \alpha u(x)\right)^{p-1} \right).
	\end{align*}
	This shows \ref{prop:MeanWurzel}.
	
	Ad~\ref{prop:MeanSgn}: If we add and subtract instead of multiplying the missing potential part, we get using \eqref{eq:MeanValue},
	\[H (\phi\circ u) (x) \geq ((\phi'\circ u)(x))^{p-1} Hu(x)+ \frac{c(x)}{m(x)}\left(\p{(\phi\circ u)(x)}- ((\phi'\circ u)(x))^{p-1}\right).\]
	By the assumptions on $\phi, u$ and $c$, the right-hand side is non-negative.
	
	The last assertion follows from \ref{prop:MeanWurzel} by setting $\phi=(\cdot)^{1/q}$ and $\alpha=q$ for some $q\geq 1$.
\end{proof}

Note that in the case of $c\geq 0$ and $\phi=(\cdot)^{1/q}$ for some $q\geq 1$, it is sufficient to have $Hu \geq 0$ on $V$ instead of $H_{b,q^{p-1}\cdot c,p,m}u \geq 0$ on $V$. 

Moreover, a corresponding result for subharmonic functions goes as follows.
\begin{corollary}
	Let $p\geq 1$, $V\sse X$, and $u\in \FF(V)$. Let $\phi\in C^2([\inf_X u,\sup_X u],\RR)$ be an increasing and convex function, i.e., $\phi', \phi'' \geq 0$. 
	Then for all $x\in V$ we have
	\begin{align*}
	L (\phi\circ u) (x) \leq ((\phi'\circ u)(x))^{p-1} Lu(x).
	\end{align*}
	In particular, if $u$ is $L$-subharmonic on $V$, then $\phi\circ u$ is $L$-subharmonic on $V$.
\end{corollary}
\begin{proof}
Mimic the proof of Proposition~\ref{prop:MeanValue}.
\end{proof}

Next, we want to show a similar result as Proposition~\ref{prop:MeanValue} for a not necessarily differentiable function $\phi$. As a downside, we will need to assume that the potential is non-negative.

But first, we start with a result which seems to be true in every quasilinear potential theory. It states that the set of non-negative supersolutions is downwards directed, i.e., it is $\wedge$-stable. 

\begin{lemma}\label{lem:infSuperharmonic}
	Let $p\geq 1$, $V\sse X$ be connected and let $S$ be a family of functions $s\in \FF(V)$ which are non-negative on $X$ and $Hs\geq f \p{s}$ on $V$ for some $f\in C(X)$. Then, the pointwise infimum $u$ of functions in $S$ is also in $\FF(V)$, non-negative on $X$, and $Hu\geq f \p{u}$ on $V$. Furthermore, $u$ is either strictly positive or equal to zero on $V\cup \partial V$.
\end{lemma}
\begin{proof}	
	It follows immediately from the definitions that $u\in C(X)$ defined via $u(x)=\inf_{s\in S}s(x)$ is in $\FF(V)$ and non-negative on $X$. 

Let $s\in S$ such that $s(x)=u(x)+\epsilon$ for some $x\in V$ and $\epsilon >0$. Then, for all $y\in X$
\[\nabla_{x,y}u +\epsilon = s(x)-u(y)\geq \nabla_{x,y}s.\]
Thus, we deduce that 
\begin{align*}
	\frac{1}{m(x)}\sum_{y\in X}b(x,y)\p{\nabla_{x,y}u +\epsilon}\geq Ls(x) 
	&\geq (f(x)-c(x)/m(x))\p{s(x)}\\
	&= (f(x)-c(x)/m(x))\p{u(x)+\epsilon}.
\end{align*}
The left-hand side is finite since
\begin{align*}
	\sum_{y\in X}&b(x,y)\abs{\nabla_{x,y}u +\epsilon}^{p-1}\\&\leq 2^{p-1}\Bigl(\abs{u(x)+\epsilon}^{p-1}\deg(x)+ \sum_{y\in X}b(x,y)\abs{u(y)}^{p-1} \Bigr)<\infty.
\end{align*}
Thus, we have dominated convergence and can take the limit $\epsilon \to 0$. This results in $Hu\geq f\p{u}$. 

We show that $u(x)=0$ yields $u=0$ in $V\cup \partial V$. Indeed, note that
\[m(x)Hu(x)=- \sum_{y\in X}b(x,y)\p{u(y)}\leq 0,\]
which is strictly less than $0$ if there exists $z\sim x$ such that $u(z)>0$. But this would give $Hu(x)<0$, which is a contradiction by the first part of the proof. Hence, for all $z\sim x$ we have $u(z)=0$. 
\end{proof}

Next, we show a discrete versions of \cite[Theorem~7.5]{HKM} repectively \cite[Section~9.8]{Bjoern}. Note that the following proposition is valid only for non-negative potentials.

\begin{proposition}\label{prop:NewSuperharmonicFromOldc>0}
	Let $c\geq 0$ on $V\sse X$ and $p\geq 1$. Let $u\in \FF(V)$ be positive on $X$ and  superharmonic on $V$. Let further $\phi\colon (0,\infty)\to [0,\infty)$ be an increasing and concave function. Then $\phi\circ u$ is also superharmonic in $V$.
	
	In particular, the functions $u^{1/q}$, $q\geq 1$, are superharmonic on $V$. 
\end{proposition}
\begin{proof}
	By the Harnack inequality, see \cite{F:GSR}, $u$ is strictly positive on $V$. If $\phi$ is concave and increasing, then we have the following identity for all $t\in (0,\infty)$
	\[\phi(t)=\inf\set{ \alpha t+\beta\colon \alpha,\beta\in \RR, \alpha \geq 0, \alpha \tau +\beta \geq \phi(\tau) \text{ for all } \tau\in (0,\infty) }.\]
	Since $\phi \geq 0$, we have $\beta \geq 0$. Moreover, because of $c\geq 0$ and $u\geq 0$ on $V$, we get $H(\alpha u +\beta)\geq 0$ on $V$. Since $\phi\circ u$ is the infimum of a set of non-negative superharmonic functions it is by Lemma~\ref{lem:infSuperharmonic} also non-negative and superharmonic. 
\end{proof}
\begin{remark}
We note that in the case of $c=0$ on $V$, the same proof holds for an increasing and convex function $\phi\colon (0,\infty)\to \RR$. Then, in the definition of such a function via supporting lines the number $\beta$ might be negative. In particular, $\log u$ is superharmonic on $V$ with respect to $L$ if $u>0 $ is superharmonic on $V$ with respect to $L$.
\end{remark}

\subsection{Coarea Formula}\label{sec:coarea}
Here we follow the ideas in \cite[Subsection~2.3]{KePiPo2} (or \cite[Lemma~9.3.5]{KPP20}) for $p=2$ and extend them to $p\geq 1$. Additionally, we weaken the assumptions from the linear $(p=2)$-case slightly. This generalised coarea formula can also be seen as the discrete analogue of the coarea formula in \cite{Versano}. 

We will also use the following boundary notation: For $V\sse X$ we denote
\[\tilde{\partial} V:=\set{(x,y)\in V\times X\setminus V : b(x,y)>0},\]
i.e., $\tilde{\partial}V$ contains all directed edges from the interior boundary of  $V$ to the exterior boundary $\partial V$. 

Note that a strictly positive proper function on an infinite set cannot take both simultaneously, its maximum and its minimum since either $0$ or $\infty$ is an accumulation point. Moreover, note that on a finite set any strictly positive function on that set is proper.

\begin{lemma}[Stokes-type formula]\label{lem:co-area}
	Let $p\geq 1$, and $V\sse X$ be non-empty. Let $0\leq u\in \FF(V)$ be non-constant on $V\cup \partial V$ and almost proper on $X$. Let the function $g\colon (\inf_{X} u, \sup_{X} u)\to [0,\infty]$ be defined via
	\[g(t):=\sum_{\substack{x,y\in X\\ u(y)< t \leq u(x)}} b(x,y)(\nabla_{x,y}u)^{p-1}.\]
	Then, for any $t_1, t_2\in (\inf_X u, \sup_X u)$ such that $t_1\leq t_2$, the set \[W_{t_1,t_2}:=\{ x\in W :  t_1< u(x)\leq t_2\}\] is finite for any $W\sse X$, and 
	\begin{align}\label{eq:gt2}	
		g(t_1)-g(t_2)= \sum_{x\in V_{t_1,t_2}}Lu(x)m(x)+\sum_{(x,y)\in \tilde{\partial} ((X\setminus  V)_{t_1,t_2})}b(x,y)(\nabla_{x,y}u)^{p-1},
	\end{align}
	where both sides may take the value $+\infty$. 
	
	In the following assume that $u=0$ on $X\setminus V$. Furthermore, in the case of  infinite $V$, assume also 
	\begin{align}\label{eq:sup<}	
\sup_{x,y\in V, x\sim y}(u(x)-u(y))<\sup_{x\in V} u(x) - \inf_{x\in V} u(x),
	\end{align}
then $g((\inf_X u, \sup_X u))\sse (0,\infty)$.

Moreover, assume additionally that  $Lu\in \ell^1(V,m)$, and
\begin{enumerate}[label=(\alph*)]
\item\label{lem:coarea1}  $u$ takes its maximum on $V$, or there exists $S > 0$ such that  for all $x\in V$ with $u(x)> S$, we have $Lu(x) \leq 0$; and
\item\label{lem:coarea2}  $u$ takes its minimum on $V$, or there exists $I > 0$ such that  for all $x\in V$ with $u(x)< I$, we have $Lu(x) \geq 0$.
\end{enumerate} 
Then, also $g\asymp 1$ on $X$.	
\end{lemma}
\begin{proof}
For $t>0$, define for any $W\sse X$,  $W_{t}:=\{x\in W : u(x)>t\}.$ Let
$t_{1},t_{2}\in (\inf_X u,\sup_X u)$ with $t_{1}\le t_{2}$. Then, $W_{t_1,t_2}=W_{t_{1}}\setminus W_{t_{2}}$. Since $u$ is almost proper on $X$, the set $W_{t_1,t_2}$ is finite for any $W\sse X$. Therefore, the
characteristic function $1_{V_{t_1,t_2}}$ of $V_{t_1,t_2}$ is in $C_{c}(V)$.

Let us abbreviate the notation further by writing
\[b_{u}(x,y):=b(x,y)\p{\nabla_{x,y}u}, \qquad x,y\in X.\]

 Since $u\in \FF(V)$,  we can apply Green's formula, Lemma~\ref{lem:GreensFormula}, which yields
\begin{align*}
\sum_{x\in V_{t_1,t_2}}Lu(x)m(x)     
&=\sum_{x\in X}1_{V_{t_1,t_2}}(x)Lu(x)m(x) \\
    &=\frac{1}{2}\sum_{x,y \in X}b_u(x,y) \nabla_{x,y} 1_{V_{t_1,t_2}}
    =\sum_{(x,y)\in \tilde{\partial} V_{t_1,t_2}} b_u(x,y).
\end{align*}

Notice that for $t_1\leq t_2$, both $g(t_1)$ and $g(t_2)$ share the sum over the set 
\[\set{(x,y)\in X\times X : u(y)< t_1 \leq t_2\leq u(x)}.\]
Thus, we conclude
\begin{align*}
g(t_1)-g(t_2)&= \sum_{\substack{x,y\in X\\ u(y)< t_1 \leq u(x)<t_2}} b_u(x,y)-\sum_{\substack{x,y\in X \\ t_1\leq u(y)< t_2 \leq u(x)}}b_u(x,y) \\
&= \sum_{\substack{x,y\in X\\ u(y)< t_1 \leq u(x)<t_2}} b_u(x,y)+\sum_{\substack{x,y\in X\\ t_1\leq u(x)< t_2 \leq u(y)}}b_u(x,y) \\
&= \sum_{(x,y)\in \tilde{\partial} V_{t_1,t_2}}b_u(x,y)+ \sum_{(x,y)\in \tilde{\partial} ((X\setminus V)_{t_1,t_2})}b_u(x,y),
\end{align*}
together with the observation before, this shows \eqref{eq:gt2}.

We turn to $g>0$: Since $u$ is non-constant and $X$ is connected, $\tilde{\partial} X_{t}$ is non-empty for all $t\in (\inf_{X} u, \sup_{X} u)$, i.e., $g>0$ on $(\inf_{X} u, \sup_{X} u)$.

Now assume that $u\in C(V)$, and we show that $g< \infty$: If $V$ is finite, then $u(x)>t$ only finitely many times and only for $x\in V$. Since $u\in F(V)$, $g<\infty$. If $V$ is infinite, then firstly note that $u\in C(V)$ implies $\tilde{\partial} ((X\setminus V)_{t_1,t_2})=\emptyset$. By using \eqref{eq:gt2}, and since $V_{t_1,t_2}$ is finite for all $t_{1},t_{2}\in (\inf_X u,\sup_X u)$, we have $g(t_{1})<\infty$  if and only if $g(t_{2})<\infty$. Secondly, note that by Inequality~\eqref{eq:sup<}, there are $t_{1},t_{2}\in (\inf_{X} u,\sup_{X} u)$ such that $t_{2}-t_{1}> \sup_{x,y\in V, x\sim y}(u(x)-u(y))$. For the choice of these $t_{1},t_{2}$, there is no vertex in $V_{t_{2}}=X_{t_2}$ that is connected to a vertex outside of $V_{t_{1}}=X_{t_2}$. Hence, $\tilde{\partial} V_{t_{2}}= \tilde{\partial} V_{t_{2}}\cap\tilde{\partial} (X\setminus V_{t_1,t_2})$.
 
 Moreover, we observe that for any $W\subseteq X$, we have $(x,y)\in \tilde{\partial} W$ if and only if $(y,x)\in \tilde{\partial} (X\setminus W)$.  Hence, we have by the considerations before and the definition of $b_{u}$ and the assumption $u\in F(V)$ that 
\begin{align*}
g(t_{2})= \sum_{(x,y)\in \tilde{\partial} V_{t_{2}}\cap\tilde{\partial} (X\setminus
V_{t_1,t_2})}\hspace{-.3cm}b_{u}(x,y)\leq \sum_{(x,y)\in \tilde{\partial} V_{t_1,t_2}}b(x,y)\abs{\nabla_{x,y}u}^{p-1}<\infty.
\end{align*}
Thus, $g< \infty$ on $(\inf_X u,\sup_X u)$.

Next, we show the boundedness of $g$ under certain additional assumptions. The lower bound follows from  \ref{lem:coarea1}, \ref{lem:coarea2} and \eqref{eq:gt2}. Here are the details: Indeed, assume it takes its maximum on $V$, then $u^{-1}([a,\max_V u])$ is finite for all $a> 0$ by the almost properness of $u$. Hence, by \eqref{eq:gt2}, $g$ changes its value only finitely many times, and thus $g\asymp 1$ in $[a,\max_V u]$. A same argument holds, if $u$ takes its minimum instead. If $u$ does not have a maximum, then $g$ might converge to zero as $t$ goes to $\sup_V u=\sup_X u$. By  \ref{lem:coarea1} and \eqref{eq:gt2}, $g$ is increasing in a neighbourhood of $\sup_V u=\sup_X u$, and thus cannot converge to zero. By using \ref{lem:coarea2} instead, a similar argument applies, when $u$ does not have a minimum.

The upper bound follows if we additionally assume that $Lu\in \ell^1(X,m)$. Then, \eqref{eq:gt2} and \ref{lem:coarea1} yields for all $t_1\leq t_2$ with $t_2\geq S$ if $u$ does not  have a maximum and $t_2=\max_V u$ in the other case,
	\begin{align*}
		g(t_1)=g(t_2) +\sum_{x\in V_{t_1,t_2}}Lu(x)m(x)\leq g(t_2)+ \sum_{x\in V}\abs{Lu(x)}m(x) <\infty.
	\end{align*}
	The calculation for \ref{lem:coarea2} is similar.
\end{proof}
We remark that many of the latter additional assumptions in Lemma~\ref{lem:co-area} are satisfied if $Lu \in C_c(X)$.

Furthermore, if $u>0$ on $V$, then Inequality~\eqref{eq:sup<} is equivalent to
\[\sup_{x,y\in V,\, x\sim y}\frac{u(x)}{u(y)}<\sup_{x,y\in V }\frac{u(x)}{u(y)},\]
which is satisfied by proper functions of bounded oscillation on infinite $V$.

Now, we prove a formula to translate calculations and estimates of infinite sums over graphs to one dimensional integrals -- the so-called coarea formula. This formula will be of fundamental importance in the proof of Theorem~\ref{thm:mainresultc<0}.
\begin{proposition}[Coarea formula]\label{prop:coarea}
Let $p\geq 1$, and $0\lneq u\in C(V)$. Let the function  $f\colon (\inf u,\sup u) \to [0,\infty)$ be Riemann integrable. Then
\begin{align}\label{eq:coarea}
\frac{1}{2}\sum_{x,y\in X}b(x,y)\p{\nabla_{x,y}u}\int^{u(x)}_{u(y)}f(t)\,\dd t = \int_{\inf u}^{\sup u}f(t)g(t)\,\dd t,
\end{align}
 where both sides can take the value $+\infty$, and  $g\colon(\inf_X u,\sup_X u)\to[0,\infty]$ is given by
\begin{align*}
g(t):=    \sum_{\substack{x,y\in X\\u(y)< t\leq u(x)}}b(x,y)(\nabla_{x,y}u)^{p-1}.
\end{align*}
Assume further that $Lu\in  \ell^1(V,m)$, and
\begin{enumerate}[label=(\alph*)]
  \item\label{prop:coareaA} $u$ is almost proper and non-constant on $V\cup\partial V$,
  \item\label{prop:coareaB} $\sup_{x,y\in V, x\sim y}(u(x)-u(y))<\sup_{x\in V} u(x) - \inf_{x\in V} u(x)$ in the case of infinite $V$,
  \item\label{prop:coareaC}  $u$ takes its maximum, or there exists $S > 0$ such that for all $x\in V$ with $u(x)> S$, we have $Lu(x) \leq 0$, and
\item\label{prop:coareaD}  $u$ takes its minimum, or there exists $I > 0$ such that for all $x\in V$ with $u(x)< I$, we have $Lu(x) \geq 0$.
\end{enumerate}
Then, $g\asymp 1$.
\end{proposition}
\begin{remark}
 Let $u$ and $f$ be as in Proposition~\ref{prop:coarea} with  $f$ being continuous, and assume $u(x)\neq u(y)$ for all $x\sim y$ with either $x\in V$ or $y\in V$. Then, we can use the mean value theorem and get that there is $\theta_{x,y}\in (u(x)\wedge u(y),u(x)\vee u(y))$ such that
\begin{align*}
    f(\theta_{x,y}) =\frac{\int^{u(x)}_{u(y)}f(t)\,\dd t}{u(x)-u(y)}.
\end{align*}
Thus, the coarea formula can be reformulated as
\begin{align*}
\frac{1}{2}\sum_{x,y\in X}b(x,y)\abs{\nabla_{x,y}u}^{p}f(\theta_{x,y})= \int_{\inf u}^{\sup u}f(t)g(t)\,\dd t.
\end{align*}
\end{remark}%
\begin{proof}[Proof of Proposition~\ref{prop:coarea}]
Let  $t>0$. As in the previous lemma, Lemma~\ref{lem:co-area}, we define $X_{t}=\{x\in X: u(x)>t\}.$
Let $1_{x,y}$ be the characteristic function of the interval
\[I_{x,y}=(u(x)\wedge u(y),u(x)\vee u(y)].\]
 Observe that $(x,y)$ or $(y,x)$ are in $\tilde{\partial} X_{t}=X_{t}\times X\setminus X_{t}$ if and only if $t\in I_{x,y}$. Using this and Tonelli's theorem, we derive
\begin{align*}
\sum_{x,y\in
X} b(x,y)\p{\nabla_{x,y}u}\int^{u(x)}_{u(y)}f(t)\,\dd t
&= \sum_{x,y\in  X}b(x,y)\abs{\nabla_{x,y}u}^{p-1}\int_{\inf u}^{\sup u}f(t)1_{x,y}(t)\,\dd t\\
&=\int_{\inf u}^{\sup u}f(t)\sum_{x,y\in X}b(x,y)\abs{\nabla_{x,y}u}^{p-1}1_{x,y}(t)\,\dd t\\
&=2\int_{\inf u}^{\sup u}f(t)\sum_{(x,y)\in \tilde{\partial} X_{t}} b(x,y)\abs{\nabla_{x,y}u}^{p-1}\,\dd t\\
&=2\int_{\inf u}^{\sup u}f(t)\sum_{(x,y)\in \tilde{\partial} X_{t}} b(x,y)\p{\nabla_{x,y}u}\,\dd t
\end{align*}
since $u(x)\ge u(y)$ for $(x,y)\in\tilde{\partial} X_{t}$. This shows the first part of the theorem. The second part follows from Lemma~\ref{lem:co-area}.
\end{proof}

\subsection{Criticality}\label{sec:criti}

We start with an auxiliary lemma. The lemma introduces a cut-off function which has its origin probably in \cite{PS05} where it was successfully used for the $p$-Laplacian on $\RR^d$. In \cite{DP16}, it was used to show criticality of $p$-Schrödinger operators in the continuum, and in \cite{KePiPo2} the same function was used to show criticality for linear Schrödinger operators on graphs. However, this cut-off function is one particular choice but others are possible as well to prove the following main results, take e.g. partially linear functions.

Let $n\in \NN$ and define the cut-off function $\psi_{n}\colon [0,\infty)\to [0,1]$ via
\begin{align}\label{eq:cutoff}
	\psi_{n}(t):=\left(  {2 + \frac{\log(t)}{\log n}}  \right) 1_{[\frac{1}{n^{2}},\frac{1}{n}]}(t)+ 1_{[\frac{1}{n},n]}(t)+ \left(  {2 - \frac{\log(t)}{\log n}}  \right)  1_{[n,n^{2}]}(t).
\end{align}
Clearly, $\psi_n\nearrow 1$ pointwise as $n\to \infty$. Moreover, we have the following estimate.

\begin{lemma}\label{lem:criticalHardy}
Let $0<\beta< \alpha<\infty$, and $\psi_n$ be the cut-off function defined in \eqref{eq:cutoff}, $n\in \NN$. Then,
\begin{align}\label{eq:criticalHardy1}
\frac{\abs{\nabla_{\alpha,\beta}\psi_n}^{p}}{(\alpha -\beta)^{p-1}} \leq \frac{\int^{\alpha}_{\beta} t^{p-1} \abs{\psi_n^{\prime}(t)}^p\dd t}{\beta^{p-1}}
\end{align}
and there is a positive constant $C$ such that for all $q> 1$,
\begin{align}\label{eq:criticalHardy2}
\frac{\bigl(\alpha^{1/q}-\beta^{1/q}\bigr)^{p}\bigl(\frac{1}{2}(\abs{\psi_n(\alpha)}+\abs{\psi_n (\beta)})\bigr)^{p}}{(\alpha -\beta)^{p-1}} \leq C  \alpha^{p/q-p+1}\int^{\alpha}_{\beta} \frac{\abs{\psi_n(t)}^p}{t}\dd t.	
\end{align}
\end{lemma}
\begin{proof} The inequalities are clearly satisfied if $\alpha \leq 1/n^2$ or $\beta>n^2$, since the left-hand sides and the right-hand sides vanish then. Thus, we can assume in the following that $\alpha > 1/n^2$ and $\beta \leq n^2$.

Ad~\eqref{eq:criticalHardy1}: 
 Firstly, we briefly show that 
\begin{align*}
\nabla_{\alpha,\beta}\psi_n \leq \frac{\nabla_{\alpha,\beta}\log}{\log n}.	
\end{align*}
This can easily be obtained by a case analysis. Note that
\[\psi_n(t)=\frac{\log (n^2t)}{\log n} 1_{[\frac{1}{n^{2}},\frac{1}{n}]}(t)+ 1_{[\frac{1}{n},n]}(t)+\frac{\log (n^2/t)}{\log n}  1_{[n,n^{2}]}(t), \qquad t\in\RR.\]
The cases $\alpha, \beta\in (0,n]$ and $\alpha,\beta \in [1/n,\infty)$ can also be obtained from \cite{KPP20}. The remaining cases follow immediately from the formula above.

Secondly, we show that 
\begin{align*}
\frac{\nabla_{\alpha,\beta}\psi_n}{\int_{\beta}^{\alpha} t^{p-1} \abs{\psi_n^{\prime}(t)}^p\dd t} \leq\log^{p-1}(n).	
\end{align*}
Indeed, by the fundamental theorem of calculus, we obtain
\begin{align*}
	\frac{\nabla_{\alpha,\beta}\psi_n}{\int_{\beta}^{\alpha} t^{p-1} \abs{\psi_n^{\prime}(t)}^p\dd t} 
=\frac{\int_{\beta}^{\alpha} \psi_n^{\prime}(t)\dd t}{\int_{\beta}^{\alpha} t^{p-1} \abs{\psi_n^{\prime}(t)}^p\dd t} 
	&= \log^{p-1}(n)\frac{\int_{\beta\vee 1/n^2}^{\alpha\wedge 1/n} 1/t \,\dd t-\int_{\beta\vee n}^{\alpha\wedge n^2} 1/t \,\dd t} {\int_{\beta\vee 1/n^2}^{\alpha\wedge 1/n} 1/t \,\dd t+\int_{\beta\vee n}^{\alpha\wedge n^2} 1/t \,\dd t}\\
	&\leq \log^{p-1}(n).
\end{align*}

Using the first two results together with
\begin{align*}
\frac{\nabla_{\alpha, \beta}\log }{\alpha-\beta} \leq \log^{\prime}(\beta) = \frac{1}{\beta},
\end{align*}
we derive at
\begin{align*}
\frac{\abs{\nabla_{\alpha,\beta}\psi_n}^{p}}{(\alpha -\beta)^{p-1}\int^{\alpha}_{\beta} t^{p-1} \abs{\psi_n^{\prime}(t)}^p\dd t} 
&=\frac{\abs{\nabla_{\alpha,\beta}\psi_n}^{p-2}\nabla_{\alpha,\beta}\psi_n}{\abs{\alpha -\beta}^{p-2}(\alpha -\beta)}\cdot \frac{\nabla_{\alpha,\beta}\psi_n}{\int^{\alpha}_{\beta} t^{p-1} \abs{\psi_n^{\prime}(t)}^p\dd t}\\
&\leq \frac{\abs{\nabla_{\alpha,\beta}\log}^{p-2}\nabla_{\alpha,\beta}\log}{\log^{p-1}(n)\abs{\alpha -\beta}^{p-2}(\alpha -\beta)}\cdot \log^{p-1}(n)\\
&\leq \frac{1}{\beta^{p-1}}.
\end{align*}
This proves the first inequality.

Ad~\eqref{eq:criticalHardy2}: By substituting $t=\beta/ \alpha\leq 1$, and since $t\mapsto (1-t^{1/q})/(1-t)$ is strictly monotonously decreasing as $t\geq 0$ increases, we have
\[\frac{(\alpha^{1/q}-\beta^{1/q})^{p}}{(\alpha - \beta)^{p}}=\alpha^{p/q-p}\cdot\Bigl(\frac{1-t^{1/q}}{1-t}\Bigr)^{p}\leq \alpha^{p/q-p}.\]
 Moreover, it is not difficult to see that there is a positive constant $C$ such that
\begin{align*}
	\bigl(\frac{1}{2}(\abs{\psi_n(\alpha)}+\abs{\psi_n (\beta)})\bigr)^{p}\leq C \int^{\alpha}_{\beta} t^{-1} \abs{\psi_n(t)}^p\dd t.
\end{align*}
The key to the above inequality is that the left-hand side is always smaller than $1$, and the right-hand side is equivalent to $\log (n)$ in the case of $\alpha> 1/n^2$ and $\beta< n^2$, and to zero elsewhere. In the latter case also the left-hand side vanishes.

Using $(\alpha- \beta) \leq \alpha$ (and $\alpha > 1/n^2$, $\beta< n^2$, $0<\beta< \alpha$), we derive at
\begin{align*}
	\frac{\bigl(\alpha^{1/q}-\beta^{1/q}\bigr)^{p}\bigl(\frac{1}{2}(\abs{\psi_n(\alpha)}+\abs{\psi_n (\beta)})\bigr)^{p}}{(\alpha -\beta)^{p-1}\int^{\alpha}_{\beta} t^{-1} \abs{\psi_n(t)}^p\dd t} 
	\leq C_p\alpha^{p/q-p}(\alpha -\beta)
	\leq C_p \alpha^{p/q-p+1}.	
\end{align*}
This proves the second inequality.
\end{proof}

Using the latter lemma and the coarea formula, we will show next that certain $p$-superharmonic functions will give us weights $w$ such that the resulting $p$-energy functional $h-w_p$ is critical. Recall that $F(V)\cap C(V)=C(V)$ if the graph is locally finite on $V\sse X$.
\begin{proposition}[Criticality]\label{prop:criticalHardyNonPositive} Let $V\sse X$ be connected and non-empty such that $(V,b|_{V\times V})$ is locally finite on $V$, and let $c$ be an arbitrary potential. Define $H:=H_{b,c,p,m}$ with corresponding subcritical $p$-energy functional $h$ and $p$-Laplacian $L$. Suppose that $u\in F(V)\cap C(V)$ is a positive function on $V$,  proper and of bounded oscillation on $V$ with $\tilde{H} u \gneq 0$ on $V$, where $\tilde{H}:=H_{b,q^{p-1}\cdot c,p,m}$, and $q:=p/(p-1).$
Furthermore, assume that $Lu\in \ell^1(V,m)$, and 
\begin{enumerate}[label=(\alph*)]
  	\item\label{prop:criticalHardyNonPositive1} $u$ takes its maximum on $V$, or  there exists $S>0$ such that for all $x\in V$ with $u(x)>S$ we have $Lu(x) \leq  0$, and
	\item\label{prop:criticalHardyNonPositive2}  $u$ takes its minimum on $V$, or  there exists $I>0$ such that for all $x\in V$ with $u(x)<I$ we have $Lu(x) \geq  0$.
\end{enumerate}
 Then $h- (wm)_p$ is critical in $V$, where $w:=H(u^{1/q})/u^{(p-1)/q}$.
\end{proposition}
\begin{proof}
By the Harnack inequality, \cite[Lemma~4.4]{F:GSR}, $u>0$ on $V$. We set $v:=u^{1/q}$. Because of Proposition~\ref{prop:MeanValue}~\ref{prop:MeanWurzel}, $v$ is  strictly $p$-superharmonic with respect to $H$ on $V$.
Furthermore, by the definition of $w$, the function $v$ is a positive $p$-harmonic function with respect to $H-w$ in $X$, i.e., $Hv=wv^{p-1}$ on $V$.

The strategy of the proof is to construct a null-sequence on $V$ with respect to $h-(wm)_p$ which converges pointwise to $u$. By Theorem~\ref{thm:charCriti}, this then implies that $h-(wm)_p$ is critical on $V$.

We take the cut-off function $\psi_{n}\colon [0,\infty )\to [0,1]$ defined in \eqref{eq:cutoff} and set $e_n\in C(V)$ via
\begin{align*}
	e_{n}=\psi_{n}\circ v \text{ on } V.
\end{align*}
If $V$ is finite then, clearly $e_n\in C_c(V)$. If $V$ is infinite, we also have $e_{n}\in C_{c}(V)$, since $\supp( \psi_{n})\subseteq (0,\infty)$, and $\sup_V u=\infty$ or $\inf_V u=0$ by the properness assumption on $u$. Obviously, $e_{n}\nearrow 1$ pointwise on $V$ as $n\to\infty$. So, we are left to show $(h-(wm)_p)(ve_{n}) \to 0$ as $n\to\infty$.

Using \eqref{eq:GSRHoelder}, we get for some positive constant $C_p$,
\begin{align*}
	(h-(wm)_p)(v {e_{n}})\leq C_p\cdot \begin{cases}h_{v,1}(e_n), &\qquad 1< p\leq 2, \\ h_{v,1}(e_n) + \bigl(\frac{h_{v,1}(e_n)}{h_{v,3}(e_n)}\bigr)^{2/p}h_{v,3}(e_n), &\qquad p>2.\end{cases}
\end{align*}
We will show that the right-hand sides vanish as $n\to \infty$.

We compute
\begin{align*}
h_{v,1}(e_n)&=\sum_{x,y\in
X}b(x,y)(u(x)u(y))^{{p}/{2q}}\abs{\nabla_{x,y}e_n}^{p}\\
&=\sum_{x,y\in X}b(x,y)\abs{\nabla_{x,y}u}^{p-2}(\nabla_{x,y}u)a_n(x,y)\left({{\int^{u(x)}_{u(y)}t^{p-1}\abs{\phi_{n}'(t)}^{p}\,\dd t}}\right),
\end{align*}
where
\begin{equation*}
    a_n(x,y) := \frac{\bigl(u(x)u(y)\bigr)^{{(p-1)}/{2}}\abs{\nabla_{x,y}e_n}^{p}}
    {\abs{\nabla_{x,y}u}^{p-2}(\nabla_{x,y}u){\int^{u(x)}_{u(y)}t^{p-1}\abs{\psi_{n}'(t)}^{p}\,\dd t}}
\end{equation*}
whenever the denominator does not vanish and $a_n(x,y)=0$ otherwise.

Using \eqref{eq:criticalHardy1}, we obtain (assuming without loss of generality that $u(x)\geq u(y)$, otherwise we use a symmetry argument)
\begin{align*}
a_n(x,y) &\leq \frac{(u(x)u(y))^{(p-1)/2}}{u^{p-1}(y)} \leq \sup_{x \sim y} \Bigg( \frac{u(x)}{u(y)} \Bigg)^{(p-1)/2} =:C_{0}< \infty,	
\end{align*}
since $u$ is of bounded oscillation on $V$.

We use this estimate and apply the coarea formula with $f(t)=t^{p-1}\abs{\psi'_n(t)}^{p}$. Note that the assumptions of the corresponding proposition, Proposition~\ref{prop:coarea}, are fulfilled.  Therefore, there exist positive constants $C_{1} $ and $C_{2}$ such that for all $n \in \NN$

\begin{align*}
h_{v,1}(e_{n})&\leq C_0\sum_{x,y\in X}b(x,y)\abs{\nabla_{x,y}u}^{p-2}(\nabla_{x,y}u)\left({{\int^{u(x)}_{u(y)}t^{p-1}\abs{\psi_{n}'(t)}^{p}\,\dd t}}\right)\\
&\leq C_{1}\int^{\sup u}_{\inf u}t^{p-1}\abs{\psi_{n}'(t)}^{p}\,\dd t\\
&\leq C_{2}\left({\frac{1}{\log n}}\right)^{p}\left({\int_{\frac{1}{n^{2}}}^{\frac{1}{n}}\frac{\,\dd t}{t}} +\int_{{n}}^{{n}^{2}}\frac{\,\dd t}{t}\right) =\frac{2C_{2}}{\log^{p-1} n}.
\end{align*}
The term on the right-hand side tends to $0$ as $n \to \infty$. Thus, for $1<p\leq 2$, $(ve_n)$ is indeed a null-sequence for $(h-(wm)_p)_v$, which implies by the ground state representation formula the criticality of the functional $h-w_p$.

We are left to analyse the case $p\geq 2$. Here we calculate
\begin{align*}
h_{v,3}(e_n)=\sum_{x,y\in X}b(x,y)\abs{\nabla_{x,y}u}^{p-2}(\nabla_{x,y}u)\tilde{a}(x,y)\left({{\int^{u(x)}_{u(y)}t^{-1}\abs{\phi_{n}(t)}^{p}\,\dd t}}\right),
\end{align*}
where
\begin{equation*}
    \tilde{a}_n(x,y) := \frac{\abs{\nabla_{x,y}u^{1/q}}^{p}\bigl(\frac{\abs{e_n(x)}+\abs{e_n (y)}}{2}\bigr)^{p}}{\abs{\nabla_{x,y}u}^{p-2}\nabla_{x,y}u\int^{u(x)}_{u(y)} t^{-1} \abs{\psi_n(t)}^p\dd t} 
\end{equation*}
whenever the denominator is non-zero and $\tilde{a}_n(x,y)=0$ otherwise.

Using \eqref{eq:criticalHardy2} with $q=p/(p-1)$, we obtain (assuming without loss of generality that $u(x)\geq u(y)$, otherwise we use a symmetry argument) that there is a positive constant $C_3$ such that
\begin{align*}
\tilde{a}_n(x,y) &\leq C_3 u^{0}(x)=C_3< \infty.	
\end{align*}
Then, we can use again the coarea formula but this time with $f(t)=t^{-1}\abs{\psi_n(t)}^{p}$. Thus, there exist positive constants $C_2, C_3, C_4$ such that 
\begin{align*}
&	h_{v,1}(e_n) + (h_{v,1}(e_n))^{2/p}(h_{v,3}(e_n))^{(p-2)/p} \\
	&\leq \frac{2C_{2}}{\log^{p-1} n} \\
	&\,+ \Bigl(\frac{2C_{2}}{\log^{p-1} n}\Bigr)^{2/p}\left(C_{3} \sum_{x,y\in X}b(x,y)\abs{\nabla_{x,y}u}^{p-2}\nabla_{x,y}u\left({{\int^{u(x)}_{u(y)}t^{-1}\abs{\psi_{n}(t)}^{p}\,\dd t}}\right) \right)^{(p-2)/p}\\
	&=  \frac{2C_{2}}{\log^{p-1} n} + \Bigl(\frac{2C_{2}}{\log^{p-1} n}\Bigr)^{2/p}\left(C_{3} {{\int^{\sup u}_{\inf u}t^{-1}\abs{\psi_{n}(t)}^{p}\,\dd t}}\right)^{(p-2)/p}.
\end{align*}
Since 
\[\int^{\sup u}_{\inf u}t^{-1}\abs{\psi_{n}(t)}^{p}\,\dd t\asymp \int_{1/n}^{n}t^{-1} \dd t \asymp \log n,\] we conclude,
\begin{align*}
	\ldots \leq  \frac{2C_{2}}{\log^{p-1} n} + C_{4}\frac{ \log^{(p-2)/p}n}{\log^{(p-1)\cdot 2/p} n} =  \frac{2C_{2}}{\log^{p-1} n}+ \frac{C_4}{\log^{(p+1)/p} n} \to 0, \quad n\to\infty.
\end{align*}
This shows the statement.
\end{proof}

\subsection{Null-criticality}\label{sec:nullcrit}
In the present subsection we prove the null-criticality of $h-w_p$ for a specific $w$ under some familiar constraints, confer \cite{KePiPo2} for $p=2$.

We need the following elementary inequality.
\begin{lemma}\label{lem:nullcriticality}
	For all $p, q \geq 1$ and $0\leq  \beta< \alpha <\infty$, we have
	\[\frac{(\alpha^{1/q}-\beta^{1/q})^{p}}{(\alpha - \beta)^{p-1}\int_{\beta}^{\alpha}1/t\,\dd t}\geq \frac{\alpha^{p/q-p}\beta}{q^{p}}.\]
\end{lemma}
\begin{proof}
If $\beta=0$ there is nothing to prove. Thus, assume $\beta >0$ in the following. Because of 
	\begin{align*}
\frac{\log \alpha - \log \beta}{\alpha-\beta} \leq \log^{\prime}(\beta) = \frac{1}{\beta},
\end{align*}
we have 
\[\frac{\alpha - \beta}{\int_{\beta}^{\alpha}1/t\,\dd t}\geq \beta.\]
Moreover, substituting $t=\beta/ \alpha\leq 1$ yields
\[\frac{(\alpha^{1/q}-\beta^{1/q})^{p}}{(\alpha - \beta)^{p}}=\alpha^{p/q-p}\cdot\Bigl(\frac{1-t^{1/q}}{1-t}\Bigr)^{p}.\]
Let $g(t)=(1-t^{1/q})/(1-t)$. Then it is easy to see that $g$ is strictly monotonously decreasing for all $t\geq 0$. Moreover, by L'H\^{o}pital's rule it follows that 
\[\min_{t\in [0,1]}g(t)=\lim_{t\to 1}g(t)=1/q.\]
Combining this with the inequalities before yields the result.
\end{proof}

Now we can show the main result of this subsection.
\begin{proposition}[Null-criticality]\label{prop:nullcritNonPositive}
	Let $V\sse X$ be connected, non-empty and infinite such that $(V,b|_{V\times V})$ is locally finite on $V$. Define $H:=H_{b,c,p,m}$ with corresponding subcritical $p$-energy functional $h$ and $p$-Laplacian $L$. Suppose that $0\lneq u\in F(V)\cap C(V)$ is a proper function of bounded oscillation in $V$ with $ \tilde{H} u\geq 0$ on $V$, where $\tilde{H}:=H_{b,q^{p-1}\cdot c,p,m}$ and $q:=p/(p-1)$. 
	Furthermore, assume that we have $Lu\in \ell^1(V,m)$, $u\in\ell^{p-1}(V,c_-)$, and
	\begin{enumerate}[label=(\alph*)]
		\item\label{prop:nullcriticalHardyNonPositive1} $u$ takes its maximum on $V$, or  there exists $S>0$ such that for all $x\in V$ with $u(x)>S$ we have $Lu(x) \leq  0$, and
		\item\label{prop:nullcriticalHardyNonPositive2}  $u$ takes its minimum on $V$, or  there exists $I>0$ such that for all $x\in V$ with $u(x)<I$ we have $Lu(x) \geq  0$.
	\end{enumerate}
	Then, the $p$-energy functional $h-(wm)_p$ with $w:=H(u^{1/q})/(u^{(p-1)/q})$, is null-critical in $V$ with respect to $wm$.
\end{proposition}
\begin{proof}  
	By Proposition~\ref{prop:criticalHardyNonPositive}, we know that $h-(wm)_p$ is critical on $V$ with Agmon ground state $u^{1/q}$. We have to show that $(wm)_p(u^{1/q})=\infty$. 
	
	Moreover, by Theorem~\ref{thm:charCriti} and Appendix~\ref{app}, there is a null sequence $(e_n^{1/q})$ to $h-(wm)_p$ on $V$ which converges pointwise and monotone increasing to $u^{1/q}$. Hence, we have $h(e_n^{1/q})-(wm)_p(e_n^{1/q})\to 0$ as $n\to \infty$. If we can show that $h(e_n^{1/q})\to \infty$, then also $(wm)_p(e_n^{1/q})\to \infty$. Since $0\leq e_n^{1/q}\leq u^{1/q}$, this would imply $(wm)_p(u^{1/q})=\infty$.
	
	If $c(x)\geq 0$ for some $x\in V$, then the potential part at $x$ of the $p$-energy functional can be bounded from below by $0$. Because of $u\in\ell^{p-1}(V,c_-)$ and $0\leq e_n^{1/q}\leq u^{1/q}$, the potential of the negative part of $c$ remains finite for all $n\in\NN$.  Altogether, we only have to consider the divergence part. Denote $K_n:=\supp e_n\in C_c(V)$. We have	
	\[\sum_{x,y\in X}b(x,y)\abs{\nabla_{x,y} e_n^{1/q}}^p= \sum_{x,y\in K_n\cup \partial K_n}b(x,y)\abs{\nabla_{x,y} e_n^{1/q}}^p.\]
	
	By Lemma~\ref{lem:nullcriticality}, we have whenever $ e_n(x)\neq e_n(y) $ for $x\in K_n$ with $x\sim y$
	\begin{align*}
	a(x,y):=\frac{\abs{\nabla_{x,y}e_n^{1/q}}^{p}}{\abs{\nabla_{x,y}e_n}^{p-2}(\nabla_{x,y}e_n)\int_{e_n(y)}^{e_n(x)}1/t\,\dd t} \ge \frac{1}{q^{p}}\inf_{x \in K_n, x\sim y}\frac{e_n(y)}{e_n(x)}=:C_{n}.
	\end{align*}
	Otherwise, we estimate $a$ by $0$ from below. Since $K_n$ is finite, $e_n$ is almost proper on $V$, and non-constant on $K_n\cup \partial K_n$. We apply the coarea formula (Proposition~\ref{prop:coarea}) with $f(t)=1/t$. 
	Thus, we get
	\begin{align*}
	\sum_{x,y\in K_n\cup \partial K_n}b(x,y)(\nabla_{x,y}e_n^{{1}/{q}})^{p}
	&=\sum_{x,y\in K_n\cup\partial K_n}b(x,y)a(x,y)\p{\nabla_{x,y}e_n}\int_{e_n(y)}^{e_n(x)}1/t \,\dd t\\
	&\geq C_{n}\sum_{x,y\in V}b(x,y)\p{\nabla_{x,y}e_n}\int_{e_n(y)}^{e_n(x)}1/t\,\dd t \\   
	&\geq \tilde{C}_n\int^{\max_{K_n} e_n}_{\min_{K_n} e_n}1/t\,\dd t,
	\end{align*}
	where $\tilde{C}_n$ is a positive constant. Since $u$ is of bounded oscillation, $\lim_{n\to \infty}\tilde{C}_n\in (0,\infty)$.  By the properness of $ u$ and since $V$ is infinite, we have that $0$ or $\infty$ are accumulation points of $u$ in $V$, and therefore $\int^{\sup_V u}_{\inf_V u}1/t\,\dd t=\infty$. Consequently, $(wm)_p(u^{{1}/{q}})=\infty$.
\end{proof}

\subsection{Optimality Near Infinity}\label{sec:optimal}
Here, we finally prove Theorem~\ref{thm:mainresultc<0}. Because of Proposition~\ref{prop:criticalHardyNonPositive} and Proposition~\ref{prop:nullcritNonPositive}, we are only left to show the optimality near infinity. This, however, is a consequence of Theorem~\ref{thm:decayImpliesOpti}.

\begin{proof}[Proof of Theorem~\ref{thm:mainresultc<0}]
	Proposition~\ref{prop:criticalHardyNonPositive} and Proposition~\ref{prop:nullcritNonPositive} imply that the function $w:=H(u^{1/q})/u^{(p-1)/q}$, $q:=p/(p-1)$, multiplied with $m$ is a $p$-Hardy weight such that $h-(wm)_p$ is null-critical with Agmon ground state $u^{1/q}$. 
	
	Furthermore, since $u$ is proper and of bounded oscillation, also $u^{1/q}$ is it. Thus, we can apply Theorem~\ref{thm:decayImpliesOpti} and get that $wm$ is indeed an optimal $p$-Hardy weight.
\end{proof}

\section{Applications}\label{appli}
Here, we briefly discuss two applications of the Hardy inequality, an uncertainty-type principle and a Rellich-type inequality. 

\subsection{Heisenberg-Pauli-Weyl-type Inequality}\label{sec:Heisi}
The famous Heisenberg-Pauli-Weyl uncertainty principle is a direct consequence of the Hardy inequality. It asserts, roughly speaking, that the position and momentum of a particle can not be determined simultaneously. For further information confer e.g. \cite[Subsection~1.6]{BEL15} for a detailed discussion in the Euclidean space, \cite{KOe09, KOe13, K18Heisenberg} for  Riemannian manifolds, or \cite[Section~3]{BGGP20} for a recent version in the hyperbolic space.

Assume that $h$ is subcritical in $X$ with Hardy weight $w$. Then by the Hölder and Hardy  inequality, we derive for all $\phi\in C_c(\supp (w))$
\begin{align*}
	\sum_{x\in \supp (w)}\abs{\phi (x)}^p&= \sum_{x\in \supp (w)} \bigl( w^{-1/p}(x)\abs{\phi(x)}^{p-1}\bigr)\bigl( w^{1/p}(x)\abs{\phi(x)}\bigr) \\
	&\leq \left( \sum_{x\in \supp (w)}  w^{-1/(p-1)}(x)\abs{\phi(x)}^{p} \right)^{(p-1)/p} \left( \sum_{x\in X}  w(x)\abs{\phi(x)}^{p}\right)^{1/p}\\
	&\leq  \left( \sum_{x\in \supp (w)}  w^{-1/(p-1)}(x)\abs{\phi(x)}^{p} \right)^{(p-1)/p} \cdot  h^{1/p}(\phi).
\end{align*}
This is a quasilinear version of the Heisenberg-Pauli-Weyl inequality on graphs.

For the special case of the line graph on $X=\NN_0$ discussed in the introduction, we thus obtain by taking the optimal Hardy weight

\begin{align*}
	w_p^{\mathrm{opt}}(n):=\left( 1-\left(1-\frac{1}{n}\right)^{\frac{p-1}{p}}\right)^{p-1}-\left( \left(1+\frac{1}{n}\right)^{\frac{p-1}{p}}-1\right)^{p-1}, \qquad n\in \NN, 
	\end{align*}
	the following sequence of inequalities on $C_c(\NN)$
	
\begin{align*}
	\sum_{n\in \NN}\abs{\phi (n)}^p
	&\leq  \left( \sum_{n\in \NN}  (w_p^{\mathrm{opt}})^{-1/(p-1)}(n)\abs{\phi(n)}^{p} \right)^{(p-1)/p} \cdot \left( \sum_{n=1}^\infty \abs{\phi(n)-\phi(n-1)}^p \right)^{1/p}\\
	&\leq  \frac{p}{p-1} \cdot \left( \sum_{n\in \NN} n^{p/(p-1)}\abs{\phi(n)}^{p} \right)^{(p-1)/p} \cdot \left( \sum_{n=1}^\infty \abs{\phi(n)-\phi(n-1)}^p \right)^{1/p},
\end{align*}
where the second inequality follows from the comparison with the classical Hardy weight $w^H_p(n)= (1-(1-p))^p\cdot (1/n^p), n\in \NN$, confer \cite{FKP} for details.	
 
\subsection{Rellich-type Inequality}
 Another inequality which has attracted a lot of attention is the Rellich inequality. For more details on the history and generalisations  of this inequality in different settings, we suggest the papers \cite{BGGP20, KPP21, KOe09, KOe13} and the monograph \cite{BEL15}, and references therein. 
 
Assume that $h$ is subcritical in $X$ with strictly positive Hardy weight $w$. Then by the Hardy inequality, Green's formula (Lemma~\ref{lem:GreensFormula}), and the Hölder inequality, we obtain for all $\phi\in C_c(\supp (w))$ 

\begin{align*}
	\sum_{x\in X}w(x)\abs{\phi (x)}^p&\leq h(\phi)= \ip{H\phi}{\phi}_X= \ip{H\phi w^{-1/p}}{\phi w^{1/p}}_{\supp (w)}	 \\
	&\leq \left( \sum_{x\in\supp(w)} w^{-1 /(p-1)}(x)\abs{H\phi(x)}^{p/(p-1)} \right)^{(p-1)/p}\cdot \norm{\phi}_{p,w}.
\end{align*}
 This implies the Rellich-type inequality
 \begin{align*}
	\sum_{x\in X}w(x)\abs{\phi (x)}^p\leq  \sum_{x\in\supp(w)} w^{-1 /(p-1)}(x)\abs{H\phi(x)}^{p/(p-1)}, \qquad \phi\in C_c(\supp (w)).
\end{align*}

We have to admit that classical Rellich inequalities have different powers of the Hardy weight on both sides of the equation. This, however, needs more effort.

For the special case of the line graph on $X=\NN_0$ discussed in the introduction, we obtain by taking the classical Hardy weight $w^H_p$, and optimal Hardy weight $w_p^{\mathrm{opt}}$ from the previous subsection,
	
\begin{multline*}
	\left(\frac{p-1}{p}\right)^{p}\sum_{n\in \NN}\left(\frac{\abs{\phi (n)}}{n}\right)^p
\leq\sum_{n\in\NN}w_p^{\mathrm{opt}}(n)\abs{\phi (n)}^p\\
\leq  \sum_{n\in\NN} w_p^{-1 /(p-1)}(n)\abs{L\phi(n)}^{p/(p-1)}	
	\leq \left(\frac{p}{p-1}\right)^{1/(p-1)} \sum_{n\in \NN} \abs{n \cdot L\phi(n)}^{p/(p-1)},
\end{multline*}
for all $\phi\in C_c(\NN)$, where 
\[L\phi(n)= \sum_{m : \abs{n-m}=1} \p{\nabla_{n,m}\phi}, \qquad n\in \NN.\]

\section{Examples}\label{examples}
Here we discuss briefly some prominent examples of graphs: the natural numbers, homogeneous trees, model graphs and the Euclidean lattice. Thereafter, we give an example of a non-locally finite graph, where the formula from Theorem~\ref{thm:mainresultc<0} does not result in an optimal Hardy weight.

\begin{example}[$\NN$]
By Corollary~\ref{cor:mainresultc>0}, we see that the $p$-Hardy weight on $\NN$ obtained in \cite{FKP} and stated in the introduction is not only an improvement of the original Hardy weight but also optimal. By Versano's lemma, Lemma~\ref{lem:opti}, we also get optimal $p$-Hardy weights for a large class of Schrödinger operators on $\NN$.
\end{example}

On $\RR^d$, $p\neq d$, with the free $p$-Laplacian, the classical Hardy weight $W(x)=(p-1/p)^p\abs{x}^{-p}$ is an optimal $p$-Hardy weight on $\RR^d\setminus \set{0}$, see \cite[p. 4]{DP16} which follows by taking simply the Green's function on $\RR^d$ as reference function. Hence, for $d=1$, we have an significant difference between the continuous and discrete model in terms of the $p$-Hardy weights. The optimal $p$-Hardy weight on $(0,\infty)$ is only the first term of the Taylor expansion of the optimal $p$-Hardy weight on $\NN$.

\begin{example}[$\TT_{d+1}$, $d\geq 2$]\label{ex:treeOptimal}
A homogeneous regular tree $\TT_{d+1}$ is  a connected tree such that every vertex has $d + 1$ neighbours, $d\geq 2$. If $x$ and $y$ are neighbours then $b(x,y)=1$. Moreover, we set $m=1$ and $c=0$, and fix the root $o\in \TT_{d+1}$.	

It is not difficult to see that the Green's function $G_o$ is given by $G_o(r)=C_p\cdot d^{-r/(p-1)}$, $r\geq 0$ for some constant $C_p>0$. The Green's function $G_o$ is proper and of bounded oscillation. Obviously, also the remaining conditions in Corollary~\ref{cor:mainresultc>0} are  fulfilled. Hence, an optimal $p$-Hardy weight is given by
\begin{align*}
		w(r)=\begin{cases} (d+1)(1-d^{-1/p})^{p-1}, &r=0 \\
		d(1-d^{-1/p})^{p-1}-(d^{1/p}-1)^{p-1}, &r>0.
		\end{cases}
	\end{align*}	
	Note that for $p=2$, this is the result obtained in \cite[Eq. (1.3)]{BSV21}. Moreover, this optimal weight is constant for $r\geq 1$ and does not converge to zero. In the case of $p=2$, $w(r)$ for $r\geq 1$ is exactly the bottom of the $\ell^2$-spectrum of the free Laplacian on $\TT_{d+1}$.
\end{example}

Real hyperbolic spaces $\HHH (\RR^d)$ are often considered to be a counterpart in the continuum to homogeneous trees. Here, similar results were obtained, see \cite{BDAGG17, BGG17} (or for the closely connected Damek-Ricci spaces see \cite{FP23}). However, using the analogue method from \cite{DP16}, one can show that the weight in the continuum is larger than a similar positive constant and converges exponentially fast to it. Hence, this is somehow the reversed observation than between $\NN$ and $(0,\infty)$.

The calculation on $\TT_{d+1}$ can be generalised easily to all subcritical model graphs. This is done next.
\begin{example}[Model graphs]
We need some notation first: The function $d\colon X\times X \to [0,\infty)$ be the combinatorial graph distance, that is, its value is the least number of edges of a path connecting two given vertices. Moreover, for some fixed vertex $o\in X$ we set \[B_r(o):= \set{x\in X : d(x,o)\leq r},\quad \text{and }\quad S_r(o):= \set{x\in X : d(x,o)= r}.\] 
The inner curvature $k_-\colon X\to [0,\infty)$ and outer curvature $k_+\colon X\to [0,\infty)$ are defined via
\[ k_\pm (x)= \frac{1}{m(x)}\sum_{y\in S_{r\pm 1}(o)}b(x,y), \quad x\in S_r(o), \qquad k_-(o)=0. \]

Let us fix $o\in X$. A graph $b$ on $X$ with potential $c$ is called \emph{model graph} with respect to $o$ if $k_\pm$ and $c/m$ are spherically symmetric functions, i.e., $k_\pm(x)=k_\pm(y)$ and $c(x)/m(x)=c(y)/m(y)$ for all $x,y\in S_r(o)$ and all $r\geq 0$. The $p$-Laplacian of spherically symmetric function $f=f(r)$ is given by
\[ Lf(0)=k_+(0)\p{\nabla_{0,1}f},  \] and \[ Lf(r)=k_+(r)\p{\nabla_{r,r+1}f}+k_-(r)\p{\nabla_{r,r-1}f}, \qquad r\geq 1.\]

If the Green's function to the free $p$-Laplacian on a locally finite model graph exists, it is given by 
\[
G_o(x)=G_0(r)= \sum_{k=r}^{\infty} \left(\frac{m(o)}{\partial B_k(o)}\right)^{1/(p-1)}, \qquad x \in S_r(o), r\geq 0.
\]
If this is the case, Corollary~\ref{cor:mainresultc>0} can be applied. Setting $g=G_o^{(p-1)/p}$, we calculate
	\[ w(r):=\frac{Lg(r)}{g^{p-1}(r)}=\begin{cases}k_+(0)\left(1-\frac{g(1)}{g(0)} \right)^{p-1} &r=0,\\ k_+(r)\left(1-\frac{g(r+1)}{g(r)} \right)^{p-1}- k_-(r)\left(\frac{g(r-1)}{g(r)}-1 \right)^{p-1}, &r\geq 1.\end{cases}\]
	Hence, $wm$ is an optimal $p$-Hardy weight.
\end{example}

\begin{example}[Free $p$-Laplacian on $\ZZ^d$, $d>p$] In \cite{Prado, M77}, flows are constructed to show that $\ZZ^d$, $d>1$, is $p$-hyperbolic if and only if $1\leq p < d$, i.e., the corresponding $p$-energy functional is subcritical if and only if $1\leq p < d$. These flows implicitly define functions $u$ which satisfy the assumptions in Corollary~\ref{cor:mainresultc>0}. Thus, an optimal $p$-Hardy weight is given by $w:=Lu^{(p-1)/p}/u^{(p-1)^2/p}$, $1<p<d$.
\end{example}

A downside of the main result of this chapter is that it  is restricted to essentially locally finite graphs. It is natural to ask if the formula also yields optimal weights in beyond that setting. Next, we shown an example that this not the case if the formula in Theorem~\ref{thm:mainresultc<0} is used with Green's functions on star graphs.

\begin{example}[Star graphs]
	Even though, we cannot apply Theorem~\ref{thm:mainresultc<0} on a star graph, we can calculate the formula and check by hand if it satisfies all desired properties. 
	
	A star graph on $X=\NN_0$ with root $0$ is a locally summable graph such that $b(0,n)> 0$ for all $n\in \NN$, and $b(m,n)=0$ for all $m, n \in \NN$, i.e., every vertex is only connected to the root $0$. Let us add any potential $c$ such that $h$ is subcritical (e.g. take a positive potential $c\gneq 0$). Let  $G_0$ denote the minimal positive Green's function with $HG_0=1_0$. Note that $HG_0=0$ and $G_0>0$ on $\NN$ implies
	\begin{align}\label{eq:starCritical}
	c(k)=b(k,0)\p{\frac{G_0(0)}{G_0(k)}-1}, \qquad k\in \NN.
	\end{align}
	Hence, by defining $c$ and $b$ properly, we get a proper Green's function of bounded oscillation (take e.g. $c(k)=1/k^3$, $b(k,0)=1/k^2$ for $k\in\NN$). Furthermore, by the equality above and since $G_0\in F$, we get $G_0\in \ell^{p-1}(X,\abs{c})$ which, in turn, implies $LG_0\in \ell^1(X,m)$. Moreover, for $k\geq 1$,
\begin{align*}
		HG_0^{(p-1)/p}(k)=\frac{G_0^ {(p-1)^2/p}(k)}{m(k)}\left(b(k,0)\p{1-\left(\frac{G_0(0)}{G_0(k)}\right)^{(p-1)/p}}+ c(k)\right).
	\end{align*}
	Hence, our candidate for a weight $w$ is defined on $\NN$ via
	\begin{align*}
		w(k)&:= \frac{HG_0^{(p-1)/p}(k)}{G_0^{(p-1)^2/p}(k)}=\frac{1}{m(k)}\left(b(0,k)\p{1-\frac{G_0^{(p-1)/p}(0)}{G_0^{(p-1)/p}(k)}}+ c(k)\right)\\
		&= \frac{b(0,k)}{m(k)}\left(\p{1-\frac{G_0^{(p-1)/p}(0)}{G_0^{(p-1)/p}(k)}}+ \p{\frac{G_0(0)}{G_0(k)}-1}\right). 
	\end{align*}
	Clearly, $w>0$ on $\NN$. Note that this defines also the value at $0$ implicitly.

	The natural candidate for a ground state of $h-(wm)_p$ is $u=G_0^{(p-1)/p}$. By Hölder's inequality and the local summability of the graph we have that $u\in \FF$. Let us set $e_n=1_{[0,n]}\cdot u$ for $n\in\NN$. Then $0\leq e_n\nearrow u$ as $n\to\infty$, and $(e_n)$ is a good candidate for a null-sequence. Note that $Hu(k)=He_n(k)=w(k)e_n^{p-1}(k)=w(k)u^{p-1}(k)$ for all $0< k\leq n$. Hence,

\begin{align*}
 		(h-(wm)_p)(e_n)&=\ip{He_n}{e_n}_{\NN_0}- (wm)_p(e_n)\\
 		&= (He_n(0)-Hu(0))u(0)m(0) \\
 		&= u^{p}(0)\sum_{k=n+1}^\infty b(0,k)\left(1-\p{1-\frac{u(k)}{u(0)}}\right).
\end{align*} 	
	Let us assume that $c\gneq 0$. By Equation~\eqref{eq:starCritical}, we get that $u(0)\geq u(k)$ for all $k\in \NN$. Moreover, by the local summability of the graph, we get
\begin{align*}
\sum_{k=n+1}^\infty b(0,k)\left(1-\left(1-\frac{u(k)}{u(0)}\right)^{p-1}\right) \leq \sum_{k=n+1}^{\infty} b(0,k)< \infty.
\end{align*}
	Since $1-\left(1-u(k)/u(0)\right)^{p-1}\geq 0$, we can use the theorem of dominated convergence, and get $(h-(wm)_p)(e_n)\to 0$ as $n\to \infty$. Hence, $h-(wm)_p$ is critical with Agmon ground state $u$.
	
	Let us turn to the null-criticality: if we show that $u\not\in \ell^p(\NN_0,wm)$, then $h-(wm)_p$ is null-critical with respect to $wm$. Since $e_n\nearrow u$, we have $(wm)_p(e_n)\nearrow (wm)_p(u)$. Since $(e_n)$ is a null sequence of $h-(wm)_p$, it suffices to have a look at $h(e_n)$. Here, we calculate
	\begin{align}
		\notag h(e_n)&=\ip{He_n}{e_n}_{\NN_0} 
		=\sum_{k=0}^nHe_n(k)e_n(k)m(k) \\
		\label{eq:star3}	&=\sum_{k=0}^nb(0,k)\abs{\nabla_{0,k}u}^p+\sum_{k=n+1}^\infty b(0,k)u^{p}(0)+ \sum_{k=0}^n c(k)u^p(k).	
	\end{align}
	Since $G_0\in \FF_{b,p}$ we get from Hölder's inequality and the local summability of the graph that $u\in \FF_{b,p+1}$, which implies that the first sum remains finite as $n\to\infty$. Furthermore, since the graph is locally summable the second sum stays finite. Because of $HG_0=1_0$ on $\NN_0$, we have \[m(0)= \sum_{n=1}^\infty c(n)G_0^{p-1}(n).\] Hence, the third sum in \eqref{eq:star3} converges monotonously to $m(0)+ c(0)G_0^{p-1}(0)\in (0,\infty)$.   Hence, $u\in \ell^p(\NN_0,wm)$, and the corresponding functional is positive-critical. 
	
	Therefore, the formula of Theorem~\ref{thm:mainresultc<0} applied to a Green's function on a star graph does not result in an optimal weight.
\end{example}

\appendix 
\section{Criticality Implies Increasing Null-Sequences}\label{app} 
Here we show that criticality implies the existence of an \emph{increasing} null sequence.

There is nothing to prove for $V$ being a singleton. Denote by $u$ the Agmon ground state on $V$. Take an arbitrary increasing exhaustion $(K_n)$ of $V$ with finite, non-empty and connected sets, and some $o\in K_1$. Without loss of generality, we can assume that $u(o)=1$. 

Let $H_n$ be the $p$-Schrödinger operator we obtain by adding $ m/n$ to the potential $c$ of $H$, $n\in \NN$, with $p$-energy functional $h_n$. Then, for all $n\in \NN$,
	\[\lambda_0(W,H_n):=\inf_{\phi\in C_c(V), \norm{\phi}^p_{p,m}=1}h(\phi)\geq 1/n >0, \qquad \emptyset \neq W\sse V.\]

By the maximum principle, \cite[Proposition~3.10]{F:AAP}, we get the existence of a sequence $(v_n)$ in $C(K_n)$ such that $H_nv_n= g_n$ on $K_n$ for any $0\leq g_n\in C(K_n)$. Assume that $0\lneq g_n\to g$ pointwise. Then, again by \cite[Proposition~3.10]{F:AAP}, $v_n>0$ on $K_n$. 

If $v_n(o)\to 0$ as $n\to\infty$, then by the Harnack inequality,  \cite[Lemma~4.4]{F:GSR}, $v_n\to 0$ on $V$ which is a contradiction unless $g=0$. However, $0$ is not a positive function. This motivates to consider $w_n:=v_n/v_n(o)$ and $\tilde{g}_n:= g_n/v_n^{p-1}(o)$ instead.

Let us set
\begin{align*}
S^+_{o}(K_n,H_n):=\set{u\in \FF(K_n) : u(o)=1, H_nu \geq 0 \text{ on } K_n, u\geq 0 \text{ on } K_n\cup \partial K_n }.
\end{align*}

Note that $(w_n)_{n\geq n_0}$ is in $S^{+}_{o}(K_{n_0}, H_{n_0})$ for any $n\geq n_0\in \NN$. This implies by the Harnack inequality also that $w_n$ cannot converge to $\infty$. Thus, by the Harnack principle, \cite[Lemma~3.2]{F:AAP}, $(w_n)$ converges pointwise to some $w\in S^{+}_{o}(K_{n_0},H_{n_0})$ for all $n_0\in \NN$, and $0\leq H_{n_0}w(x) = Hw(x)+w^{p-1}(x)/n_0$ for all $x\in K_{n_0}$. Letting $n_0\to\infty$, we see that the limit $w$ is positive and $p$-superharmonic on $V$. By Theorem~\ref{thm:charCriti}, this is a contradiction unless $w=u$. Note that this implies the convergence of $(\tilde{g}_n)$ to $0$.

Our first candidate of an increasing null-sequence is $(w_n)$. Let us try to apply the weak comparison principle, \cite[Proposition~5.3]{F:AAP}. Obviously, $w_{n}\leq w_k$ on $X\setminus K_{n}$, and $k\geq n$. 
Moreover, for all $x\in K_{n}$, and $k\geq n$,
\begin{align*}
H_nw_{n}(x)- H_nw_k(x)= \tilde{g}_{n}(x)-\tilde{g}_k(x)- \left(\frac{1}{n}-\frac{1}{k} \right)w_k^{p-1}(x).
\end{align*}
Because of $w_k\in S^+_o(K_{n},H_{n})$, let us now choose the specific function $g_n= 1_o$. By the weak comparison principle, $(v_n)$ is increasing, and by the first part, it converges pointwise to $\infty$. Then, there is a subsequence $(\tilde{g}_{n_k})$ such that for all $x\in K_n$, $n_k<n_m$,
\[ \tilde{g}_{n_k}(x)-\tilde{g}_{n_m}(x) \leq \left(\frac{1}{n_k}-\frac{1}{n_m} \right)w_{n_m}^{p-1}(x).\]  
Thus, $(w_{n_k})$ is increasing by the weak comparison principle. 

By construction,  $0\leq g_{n_k}w_{n_k}\to 0\cdot u =0$ pointwise on $V$. Hence, there is a decreasing sub-subsequence $( g_{n_{k_l}}w_{n_{k_l}})$. Thus, we conclude with the aid of Green's formula and monotone convergence,
\[ 0\leq h(w_{n_{k_l}})\leq h_{n_{k_l}}(w_{n_{k_l}})= \ip{g_{n_{k_l}}}{w_{n_{k_l}}}_V\to 0, \quad n\to \infty. \]
Hence, we have an increasing null sequence.

 \noindent\textbf{Acknowledgements.} The author wishes to express his sincere gratitude to Matthias Keller for introducing him to the topic. Furthermore, the author thanks Yehuda Pinchover and Andrea Adriani for  helpful comments. The main parts of this paper were written when the authors was a doctoral student at the University of Potsdam, and was supported by the Heinrich Böll foundation (grant nr. P139140).

\noindent\textbf{Declarations.}

\noindent\textbf{Competing Interests.} The author has no competing interests to declare that are relevant to the content of this article.

\noindent\textbf{Data Availability Statement.} Data sharing is not applicable to this article as no datasets were generated or analysed during the current study.

\printbibliography
\end{document}